\documentclass[11pt]{article}
\usepackage[centertags]{amsmath}
\usepackage{amsfonts}
\usepackage{amssymb}
\usepackage{amsthm,color}
\usepackage{newlfont}
\newtheorem{Theorem}{Theorem}[section]
\newtheorem{Definition}[Theorem]{Definition}
\newtheorem{Proposition}[Theorem]{Proposition}
\newtheorem{Lemma}[Theorem]{Lemma}
\newtheorem{Corollary}[Theorem]{Corollary}
\newtheorem{Remark}[Theorem]{Remark}

\newtheorem{Hypothesis}{Hypothesis}

\numberwithin{equation}{section}

\begin{document}

\def\le{\left}
\def\r{\right}
\def\cost{\mbox{const}}
\def\a{\alpha}
\def\d{\delta}
\def\ph{\varphi}
\def\e{\epsilon}
\def\la{\lambda}
\def\si{\sigma}
\def\La{\Lambda}
\def\B{{\cal B}}
\def\A{{\mathcal A}}
\def\L{{\mathcal L}}
\def\O{{\mathcal O}}
\def\bO{\overline{{\mathcal O}}}
\def\F{{\mathcal F}}
\def\K{{\mathcal K}}
\def\H{{\mathcal H}}
\def\D{{\mathcal D}}
\def\C{{\mathcal C}}
\def\M{{\mathcal M}}
\def\N{{\mathcal N}}
\def\G{{\mathcal G}}
\def\T{{\mathcal T}}
\def\R{{\mathcal R}}
\def\I{{\mathcal I}}

\def\bw{\overline{W}}
\def\phin{\|\varphi\|_{0}}
\def\s0t{\sup_{t \in [0,T]}}
\def\lt{\lim_{t\rightarrow 0}}
\def\iot{\int_{0}^{t}}
\def\ioi{\int_0^{+\infty}}
\def\ds{\displaystyle}
\def\pag{\vfill\eject}
\def\fine{\par\vfill\supereject\end}
\def\acapo{\hfill\break}

\def\beq{\begin{equation}}
\def\eeq{\end{equation}}
\def\barr{\begin{array}}
\def\earr{\end{array}}
\def\vs{\vspace{.1mm}   \\}
\def\rd{\reals\,^{d}}
\def\rn{\reals\,^{n}}
\def\rr{\reals\,^{r}}
\def\bD{\overline{{\mathcal D}}}
\newcommand{\dimo}{\hfill \break {\bf Proof - }}
\newcommand{\nat}{\mathbb N}
\newcommand{\E}{\mathbb E}
\newcommand{\Pro}{\mathbb P}
\newcommand{\com}{{\scriptstyle \circ}}
\newcommand{\reals}{\mathbb R}

\title{A basic identity for Kolmogorov operators in the space of continuous functions related to RDEs with multiplicative noise}

\author{Sandra Cerrai\thanks{Partially supported by the NSF grant DMS0907295 ``Asymptotic Problems for SPDE's''.}
\\\normalsize University of Maryland, College Park, MD 20742-4015
\and
Giuseppe Da Prato\\
\normalsize Scuola Normale Superiore, 56126, Pisa, Italy
}

\maketitle

\begin{abstract}
We consider the Kolmogorov operator associated with a reaction-diffusion equation having polynomially growing reaction coefficient and perturbed by a noise of multiplicative type, in the Banach space $E$ of continuous functions. By analyzing the smoothing properties of the associated transition semigroup, we prove a modification of the classical  { \em identit\'e du carr\'e di champs} that applies to the present non-Hilbertian setting. As an application of this identity, we construct the Sobolev space $W^{1,2}(E;\mu)$, where $\mu$ is an invariant measure for the system, and we prove the validity of the Poincar\'e inequality and of the spectral gap.

\end{abstract}

\section{Introduction}
\label{sec1}

In the present paper we are concerned with the analysis of the Kolmogorov operator associated with the following reaction-diffusion equation in the interval $(0,1)$, perturbed by a noise of multiplicative type
\begin{equation}
\label{eq1}
\le\{\begin{array}{l}
\ds{\frac{\partial u}{\partial t}(t,\xi)=\frac{\partial^2 u}{\partial \xi^2}(t,\xi)+f(\xi,u(t,\xi))+g(\xi,u(t,\xi))\frac{\partial w}{\partial t}(t,\xi),\ \ \ \ t\geq 0,\ \ \ \xi \in\,[0,1],}\\
\vs
\ds{u(t,0)=u(t,1)=0,\ \ \ \ \ u(0,\xi)=x(\xi),\ \ \ \xi \in\,[0,1].}
\end{array}\r.
\end{equation}
Here $\partial w/\partial t(t,\xi)$ is a space-time white noise. The non-linear terms
$f,g:[0,1]\times \reals\rightarrow \reals$ are both 
 continuous, the mapping
$g(\xi,\cdot):\reals\rightarrow \reals$ is Lipschitz-continuous,
uniformly with respect to $\xi \in\,[0,1]$, without any  restriction
on its linear growth, and  the mapping $f(\xi,\cdot)$ has
polynomial growth, is locally Lipschitz-continuous and satisfies
suitable dissipativity conditions, uniformly with respect to $\xi
\in\,[0,1]$. The example of $f(\xi,\cdot)$  we have in mind  is an odd-degree polynomial, having negative leading coefficient.

In \cite{cerrai}, the well-posedness of equation \eqref{eq1} has been studied and it has been proved that for any initial datum $x \in\,E:=C_0([0,1])$ there exists a unique {\em mild} solution $u^x \in\,L^p(\Omega;C([0,T];E))$, for any $T>0$ and $p\geq 1$. This allows us to introduce the {\em Markov transition semigroup} $P_t$ associated with
 equation (\ref{eq1}), by setting for any Borel measurable
and bounded function $\varphi:E\rightarrow \reals$
\[P_t\varphi\,(x)=\E\,\varphi\,(u^x(t)),\ \ \ \ \ t\geq 0,\ \ \ x
\in\,E,\] where $u^x(t)$ is the solution of (\ref{eq1}) starting from $x$ at time $t=0$.
Hence, by proceeding as in \cite{cerrai94}, we can define the Kolmogorov operator  $\mathcal K$ associated with the semigroup $P_t$  in terms of its
 Laplace transform 
\begin{equation}
\label{e1.2t}
(\lambda-\mathcal K)^{-1}\varphi(x)=\int_0^\infty e^{-\lambda t}P_t\varphi(x)dt,\ \ \ \ \varphi\in C_b(E)\footnote{The space  of all uniformly continuous and bounded real-valued mappings defined on $E$.}.
\end{equation}

In this paper we are going to study some important properties of the Kolmogorov operator  $\mathcal K$ in $C_b(E)$ .
If we write equation \eqref{eq1} in the abstract form  
$$du(t)=\le[A u(t)+F(u(t))\r]\,dt+G(u(t))dw(t),$$
  (see  Section \ref{sec2}  below for all notations)
  then $\mathcal K$ reads formally as 
\begin{equation}
\label{e1.2s}
\mathcal K\varphi=\frac12\,\sum_{h=1}^\infty D^2\varphi(G(x)e_k,G(x)e_k)+\langle Ax+F(x), D\varphi(x) \rangle_E 
\end{equation}
(here $D\varphi$ and $D^2\varphi$ represent the first and the second derivatives of a twice differentiable function $\varphi:E\to \reals$ and  $\langle \cdot, \cdot \rangle_E$ is the duality between $E$ and its topological dual $E^*$).
Notice however that it is not easy to decide whether  a given function belongs to the domain of $\mathcal K$  or not, as it is  defined in an  abstract way by formula \eqref{e1.2t}. Our main concern here is studying some relevant properties of ${\cal K}$, such as the possibility to define the Sobolev space $W^{1,2}(E,\mu)$, with respect to the invariant measure $\mu$ for equation \eqref{eq1}, or the validity of the Poincar\'e inequality and of  the spectral gap, which, as well known, implies the exponential convergence to equilibrium. 

In the case of additive noise, that is when  $G(x)$ is constant, it is possible to study equation \eqref{eq1}  in the Hilbert space $H=L^2(0,1)$ in a generalized sense, so that the  associated  transition semigroup and Kolmogorov operator can be introduced. In this case, it has been proved that the so-called { \em identit\'e du carr\'e des champs}
\begin{equation}
\label{e1.3s}
\mathcal K(\varphi^2)=2\varphi\;\mathcal K\varphi+|G^*D\varphi|_H^2
\end{equation}
 is valid for functions $\varphi$ in a {\em core} of ${\cal K}$, that is a    subset of $D(\mathcal K)$ which is dense in the graph norm of $\mathcal K$ (see \cite{DDG}).
Identity \eqref{e1.3s} has several important consequences. Actually, if there exists an invariant measure $\mu$ for $u^x(t)$, identity \eqref{e1.3s} provides  the starting point to define the Sobolev space
$W^{1.2}(H,\mu)$. Moreover, under some additional conditions, it allows to prove the Poincar\'e inequality and the exponential convergence of $P_t\varphi$ to equilibrium ({\em spectral gap}).

\medskip

The case we are dealing with in the present paper is much more delicate, as  we are considering a polynomial reaction term $f$ combined with a multiplicative noise. Because of this, it seems better and more natural to work in the Banach space $E$ of continuous functions vanishing at the boundary, instead of in $H$. Moreover, the space $C_b(E)$ is larger than the space $C_b(H)$, and working in $C_b(E)$ allows to estimate some interesting functions as, for example, the evaluation functional $\E\d_{\xi_0}(u)=\E u(\xi_0)$, for $\xi_0 \in\,[0,1]$ fixed. 

On the other hand, deciding to work in $C_b(E)$ instead of $C_b(H)$ has some relevant consequences and there is some price to pay. In our case it means in particular that formula \eqref{e1.3s} has to be changed in a suitable way, as for $\varphi \in\,C^1_b(E)$ and $x \in\,E$ we cannot say that $D\varphi(x) \in\,H$ and hence the term $|G^*(\cdot)D\varphi|_H$ is no more meaningful. In fact, it turns out that formula \eqref{e1.3s} has to be replaced by the formula
\begin{equation}
\label{fine1}
\mathcal K(\varphi^2)=2\varphi\;\mathcal K\varphi+\sum_{k=1}^\infty\le|\le<G(\cdot)e_k,D\varphi\r>_E\r|^2,
\end{equation}
where $\{e_k\}_{k \in\,\nat}$ is the complete orthonormal system given by the eigenfunctions of the second derivative, endowed with Dirichlet boundary conditions.

Notice that, in order to give a meaning to \eqref{fine1}, for $\varphi \in\,D({\cal K})$, we have to prove that:
\medskip

(i)  $D(\mathcal K)$ is included in $C^1_b(E)$;\medskip

(ii) the series in \eqref{fine1} is convergent for any $\varphi \in\,D({\cal K})$; \medskip

(iii) $\varphi^2 \in\,D({\cal K})$, for any  $\varphi\in D(\mathcal K)$ and \eqref{fine1} holds.

\medskip

The proof of each one of these steps is very delicate in the framework we are considering here and requires the use of different arguments and techniques, compared to \cite{DDG} and \cite[Chapter 6 and 7]{tesi}.

In order to approach (i), we have proved  that the solution $u^x(t)$ of equation \eqref{eq1} is differentiable with respect to $x \in\,E$. Moreover,   we have  proved that the second derivative equation is solvable and  suitable bounds for the solution have been given. These results were not available in the existing literature and, in order to be proved, required some new arguments, based on positivity, as the classical techniques did not apply, due to the fact that $f^\prime$ is not globally bounded and the noise is multiplicative. Next, we had to prove that, as in the Hilbertian case, a Bismuth-Elworthy-Li formula holds for the derivative of the semigroup. Such a formula, as well known, provides the important gradient estimate
\[\sup_{x \in\,E}|D(P_t\varphi)|_{E^\star}\leq c\,\le(t\wedge 1\r)^{-\frac 12}\,\sup_{x \in\,E}|\varphi(x)|,\ \ \ \ t > 0.\]

In order to prove (ii), we couldn't proceed directly as in \cite[Chapter 5]{tesi}, by using the mild formulation of the first derivative equation and the fact that $e^{tA}$ is Hilbert-Schmidt, for any $t>0$, again because of the combination of the polynomial non-linearity $f$ combined with the multiplicative noise. Nevertheless, by using a duality argument, we could prove that
\[\sum_{k=1}^\infty\le|\le<G(x)e_k,D(P_t\varphi)(x)\r>_E\r|^2\leq c\,|G(x)|_E^2\|\varphi\|_0^2\le(t\wedge 1\r)^{-1},\ \ \ \ t>0,\]
and this allowed us to prove that the series in \eqref{fine1} is convergent, for any $\varphi \in\,D({\cal K})$.
To this purpose, we would like to mention the fact that our duality argument does work because we are dealing together with the two concrete spaces $E=C_0([0,1])$ and $H=L^2(0,1)$, where we can use some nice approximation and duality argument between the corresponding spaces of continuous functions $C_b(E)$ and $C_b(H)$ and the corresponding spaces of differentiable functions $C^1_b(E)$ and $C^1_b(H)$ (see Lemma \ref{Lbr21}).

Finally, in order to prove (iii), we had to use a suitable modification of the It\^o formula that applies to Banach spaces and a suitable approximation argument based on the use of the Ornstein-Uhlenbeck semigroup in the Banach space $E$.

\medskip

As mentioned before, as a consequence of  the modified  { \em identit\'e du carr\'e des champs} \eqref{fine1}, we could construct the space $W^{1,2}(E;\mu)$ and we could prove the Poincar\'e inequality and the existence of a spectral gap. To this purpose, we would like to stress that in spite of the fact that the  { \em identit\'e du carr\'e des champs} has to be modified and we have to replace $|G^\star(\cdot)D\varphi|_H^2$ by the series
\[\sum_{k=1}^\infty\le|\le<G(\cdot)e_k,D\varphi\r>_E\r|^2,\]
the Poincar\'e inequality proved is very similar to  what we have in the case of the Hilbert space $H$, that is
\[ \int_E|\varphi(x)-\overline\varphi|^2d\mu(x)\leq\rho\int_E|D\varphi(x)|_{E^\star}^2d\mu(x).\]

\section{Preliminaries}
\label{sec2}

We shall denote by $H$ the Hilbert space $L^2(0,1)$, endowed with the usual scalar product $\le<\cdot,\cdot\r>_H$ and the corresponding norm $|\cdot|_H$. Moreover, we shall denote by $E$ the Banach space $C_0([0,1])$ of continuous functions on $[0,1]$, vanishing at $0$ and $1$, endowed with the sup-norm $|\cdot|_E$ and the duality $\le<\cdot,\cdot\r>_E$ between $E$ and  its dual topological space $E^*$.

Now, if we fix  $x \in\,E$ there exists $\xi_x \in\,[0,1]$ such that $|x(\xi_x)|=|x|_{E}$. Then, if $\d$ is any element of $E^\star$ having
norm equal $1$,  the element $\d_x \in\,E^\star$ defined by
\begin{equation}
\label{bds92} \le<y,\d_x\r>_{E}:=\le\{
\begin{array}{ll}
\ds{\frac {x(\xi_x)y(\xi_x)}{|x|_{E}}}  &
\ds{\mbox{if}\ x\neq 0}\\
&\vs \ds{\le<y,\d\r>_{E}}  &\ds{\mbox{if}\ x=0,}
\end{array}\r.
\end{equation}
belongs to the subdifferential $\partial\,|x|_{E}:=\le\{\,h^\star \in\,E^\star;\
|h^\star|_{E^\star}=1,\ \le<h,h^\star\r>_E=|h|_E\,\r\}$ (see e.g.
\cite[Appendix A]{tesi} for all definitions and details).
\medskip

Next, let $X$ be a separable Banach space. ${\cal L}(X)$ shall denote the Banach algebra of all linear bounded operators in $X$ and  ${\cal L}^1(X)$ shall denote the subspace of trace-class operators. We recall that
\[\|T\|=\sup_{|x|_X\leq 1}|Tx|_X, \ \ \ \ T \in\,{\cal L}(X).\]

For any other Banach space $Y$, we denote by $B_b(X,Y)$ the linear space of all bounded and measurable mappings $\varphi:X\to Y$ and by $C_b(X,Y)$  the subspace of continuous functions. Endowed with the sup-norm
\[\|\varphi\|_0=\sup_{x \in\,X}|\varphi(x)|_Y,\]
$C_b(X,Y)$ is a Banach space. Moreover, for any $k \geq 1$, $C^k_b(X,Y)$ shall denote the subspace of all functions which are $k$-times Fr\'echet differentiable. $C^k_b(X,Y)$, endowed with the norm 
\[\|\varphi\|_k=\|\varphi\|_0+\sum_{j=1}^k\sup_{x \in\,X}|D^j\varphi(x)|=:\|\varphi\|_0+\sum_{j=1}^k\,[\varphi]_j,\]
is a Banach space. In the case $Y=\reals$, we shall set $B_b(X,Y)=B_b(X)$ and $C^k_b(X,Y)=C^k_b(X)$, $k\geq 0$.

\medskip

In what follows, we shall denote by $A$ the linear operator
\[Ax=\frac{\partial ^2 x}{\partial \xi^2},\ \ \ \ x \in\,D(A)=H^2(0,1)\cap H^1_0(0,1).\]
$A$ is a non-positive and self-adjoint operator which generates an analitic semigroup $e^{tA}$, with dense domain in $L^2(0,1)$. The space $L^1(0,1)\cap L^\infty(0,1)$ is invariant  under $e^{tA}$, so
that $e^{tA}$ may be extended to a non-positive one-parameter
contraction semigroup $e^{t A_p}$ on $L^p(0,1)$, for all $1\leq
p\leq \infty$. These semigroups are strongly continuous, for $1\leq
p<\infty$, and are consistent, in the sense that
$e^{t A_p}x=e^{t A_q}(t)x$, for all $x \in\,L^p(0,1)\cap
L^q(0,1)$. This is why we shall denote all $e^{t A_p}$ by
$e^{tA}$.  Finally, if we consider the part
of $A$ in  $E$, it generates a strongly continuous 
analytic semigroup.

For any $k \in\,\nat$, we define
\begin{equation}
\label{basis}
e_k(\xi)=\sqrt{\frac 2\pi}\sin k\pi \xi,\ \ \ \ \ \xi \in\,[0,1].\end{equation}
The family $\{e_k\}_{k \in\,\nat}$ is a complete orthonormal system in $H$ which diagonalizes $A$, so that
\[A e_k=-k^2 \pi^2 e_k,\ \ \ k \in\,\nat.\]
Notice that for any $t>0$ the semigroup $e^{tA}$ maps $L^p(0,1)$ into $L^q(0,1)$, for any $1\leq p\leq q\leq \infty$ and
and for any $p\geq 1$ there exists $M_p>0$ such that
\begin{equation}
\label{br100}
\|e^{tA}\|_{{\cal L}(L^p(0,1),L^q(0,1))}\leq M_{p,q}\,e^{-\omega_{p,q} t}t^{-\frac{q-p}{2pq}},\ \ \ \ t> 0.
\end{equation}

Here we are assuming that $\partial w(t)/\partial t$ is a space-time white noise defined on the  stochastic basis
$(\Omega,\F,\F_t,\Pro)$. Thus,  $w(t)$ can be written formally as
\[w(t):=\sum_{k=1}^\infty e_k \beta_k(t),\ \ \ \ \ t\geq 0,\]
where $\{e_k\}_{k \in\,\nat}$ is the complete orthonormal system in $H$ which
diagonalizes $A$ and $\{\beta_k(t)\}_{k \in\,\nat}$ is a sequence of mutually
independent standard real Brownian motions  on
$(\Omega,\F,\F_t,\Pro)$. As well known, the series above does not
converge in $H$, but it does converge in any Hilbert space $U$
containing $H$, with Hilbert-Schmidt embedding.

\medskip

Concerning the nonlinearities $f$ and $g$, we assume that they are both continuous. Moreover, they  satisfy the following conditions.

\begin{Hypothesis}
\label{H1}

\begin{enumerate}

\item For any $\xi \in\,[0,1]$, both $f(\xi,\cdot)$ and $g(\xi,\cdot)$ belong to $C^2(\reals)$.
\item There exists $m\geq 1$ such that for $j=0,1,2$
\begin{equation}
\label{h11}
\sup_{\xi \in\,[0,1]}|D_\rho^j f(\xi,\rho)|\leq c_j(1+|\rho|^{(m-j)^+}).
\end{equation}
\item There exists $\la \in\,\reals$ such that
\begin{equation}
\label{h12}
\sup_{(\xi,\rho) \in\,[0,1]\times \reals}f^\prime(\xi,\rho)\leq \la.
\end{equation}
\item The mapping $g(\xi,\cdot):\reals\to \reals$ is Lipschitz continuous, uniformly with respect to $\xi \in\,[0,1]$, and
\begin{equation}
\label{h13}
\sup_{\xi \in\,[0,1]}|g(\xi,\rho)|\leq c\le(1+|\rho|^{\frac 1m}\r).
\end{equation}
\item If there exist $\a>0$ and $\beta\geq 0$ such that
\begin{equation}
\label{h14}
\le(f(\xi,\rho+\si)-f(\xi,\rho)\r)\si\leq -\a\,|\si|^{m+1}+\beta\,(1+|\rho|^{m+1}),\end{equation}
then no restriction is assumed on the linear growth of $g(\xi,\cdot)$.

\end{enumerate}

\end{Hypothesis}

In what follows, for any $x, y \in\,E$ and $\xi \in\,[0,1]$ we shall denote
\[F(x)(\xi)=f(\xi,x(\xi)),\ \ \ \ [G(x)y](\xi)=g(\xi,x(\xi))y(\xi).\]

Due to  \eqref{h11}, $F$ is well defined and 
continuous from $L^{p}(0,1)$ into $L^{q}(0,1)$, for any $p,
q \geq 1$ such that $p/q\geq m$. In particular, if $m\neq 1$ then
$F$ is not  defined from $H$ into itself. Moreover, due
to \eqref{h12}, for $x, h \in\,L^{2m}(0,1)$ 
\begin{equation}
\label{dissiF} \le<F(x+h)-F(x),h\r>_{H}\leq
\la\,|h|_H^2.
\end{equation}
Clearly, $F$ is also well defined in $E$ and it is possible to prove that it is twice continuously differentiable in $E$, with 
\[[D F(x)y](\xi)=D_\rho f(\xi,x(\xi))y(\xi),\ \ \ \ [D^2 F(x)(y_1,y_2)](\xi)=D_\rho^2 f(\xi,x(\xi))y_1(\xi) y_2(\xi).\]
In particular, for any $x \in\,E$
\begin{equation}
\label{sa3} |D^j F(x)|_{{\cal L}^j(E)}\leq c\,(1+|x|^{(m-j)^+}_{E}).
\end{equation}
Moreover,  for any
$x,h \in\,E$
\begin{equation}
\label{dissiFEbis} \le<F(x+h)-F(x),\d_{h}\r>_{E}\leq
\la\,|h|_E,
\end{equation}
where  $\d_{h}$ is the element of $\partial\, |h|_E$ defined above in \eqref{bds92}.

Finally, in the case \eqref{h14} holds, we have
\begin{equation}
\label{dissiE}
\le<F(x+h)-F(x),\d_h\r>_E\leq -\a \,|h|_E^m+\beta\,\le(1+|x|_E^m\r).
\end{equation}

Next, concerning the operator $G$, as the mapping $g(\xi,\cdot):\reals\rightarrow \reals$ is
Lipschitz-continuous,  uniformly with respect to $\xi \in\,[0,1]$, the operator $G$ is Lipschitz-continuous from $H$ into
$\L(H;L^1(0,1))$, that is
\begin{equation}
\label{bds2} \|G(x)-G(y)\|_{\L(H;L^1(0,1))} \leq
c\,|x-y|_{H}.
\end{equation}
In the same way it is possible to show that the operator
$G$ is Lipschitz-continuous from $H$ into
$\L(L^\infty(0,1);H)$ and
\begin{equation}
\label{bds1} \|G(x)-G(y)\|_{\L(L^\infty(0,1);H)}\leq
c\,|x-y|_{H}.
\end{equation}

\bigskip

By proceeding similarly as in  \cite[Proposition 6.1.5]{tesi}, it is possible to prove the following result.
\begin{Lemma}
\label{Lbr21}
For any $\varphi \in\,C_b(E)$ there exists a sequence $\{\varphi_n\}_{n \in\,\nat}\subset C_b(H)$ such that
\begin{equation}
\label{br21}
\le\{
\begin{array}{l}
\ds{\lim_{n\to\infty}\varphi_n(x)=\varphi(x),\ \ \ \ x \in\,E,}\\
\vs
\ds{\sup_{x \in\,H}|\varphi_n(x)|\leq \sup_{x \in\,E}|\varphi(x)|,\ \ \ \ n \in\,\nat.}
\end{array}\r.
\end{equation}
Moreover, if $\varphi \in\,C^k_b(E)$, we have $\{\varphi_n\}_{n \in\,\nat}\subset C^k_b(H)$
and for any $j\leq k$
\begin{equation}
\label{br22}
\le\{
\begin{array}{l}
\ds{\lim_{n\to\infty}D^j\varphi_n(x)(h_i,\ldots,h_j)=D^j\varphi(x)(h_i,\ldots,h_j),\ \ \ \ x,h_1,\ldots,h_j \in\,E,}\\
\vs
\ds{\sup_{x \in\,H}|D^j\varphi_n(x)|_{{\cal L}^j(E)}\leq \sup_{x \in\,E}|D^j\varphi(x)|_{{\cal L}^j(E)},\ \ \ \ n \in\,\nat.}
\end{array}\r.
\end{equation}
\end{Lemma}
\begin{proof}
For any $x \in\,H$ we define
\[x_n(\xi)=\frac n2\int_{\xi-\frac 1n}^{\xi+\frac 1n}\hat{x}(\eta)\,d\eta,\]
where $\hat{x}(\xi)$ is the extension by oddness of $x(\xi)$, for $\xi \in (-1,0)$ and $\xi \in\,(1,2)$. Clearly, due to the boundary conditions, we have that $x_n \in\,E$, for any $n \in\,\nat$, and
\begin{equation}
\label{br20}
|x_n|_E\leq \frac n2\sup_{\xi \in\,[0,1]}\int_{\xi-\frac 1n}^{\xi+\frac 1n}|\hat{x}(\eta)|\,d\eta\leq \sqrt{\frac n2}\,|x|_H.
\end{equation}
Now, for any $n \in\,\nat$ we define
\[\varphi_n(x)=\varphi(x),\ \ \ \ x \in\,H.\]
Due to \eqref{br20}, we have that if $\varphi \in\,C_b(E)$, then $\varphi_n \in\,C_b(H)$ and
\[\sup_{x \in\,H}|\varphi_n(x)|=\sup_{x \in\,H}|\varphi(x_n)|\leq \sup_{x \in\,E}|\varphi(x)|.\]
Finally, if $x \in\,E$, then we have that $x_n$ converges to $x$ in $E$, and then we can conclude that \eqref{br21} holds.

Now, we prove \eqref{br22} for $k=1$. If $\varphi \in\,C^1_b(E)$, then for any $n \in\,\nat$ and $x, h \in\,H$ we have
\[\begin{array}{l}
\ds{\varphi_n(x+h)-\varphi_n(x)=\varphi(x_n+h_n)-\varphi(x_n)=\le<h_n,D\varphi(x_n)\r>_E+o(|h_n|_E).}
\end{array}\]
Due to \eqref{br20} we have that $o(|h_n|_E)=o(|h|_H)$ and then we can conclude that $\varphi_n$ is differentiable in $H$ and
\[\le<h,D\varphi_n(x)\r>_H=\le<h_n,D\varphi(x_n)\r>_E,\ \ \ \ x,h \in\,H.\]
As $x_n$ and $h_n$ converge respectively to $x$ and $h$ in $E$, if $x, h \in\,E$, this implies that \eqref{br22} holds.
\end{proof}

\subsection{The approximating Nemytskii operators}
\label{subsec2.1}

Let $\gamma$ be a function in $C^\infty(\reals)$ such that
\begin{equation}
\label{ro}
\gamma(x)=x,\ \ |x|\leq 1,\ \ \ \ \ |\gamma(x)|=2,\ \ |x|\geq 2,\ \ \ \ \ |\gamma(x)|\leq |x|,\ \ x \in\,\reals,\ \ \ \ \ \gamma^\prime(x)\geq 0,\ \ x\in\,\reals.
\end{equation}
For any $n \in\,\nat$, we define
\[f_n(\xi,\rho)=f(\xi,n\gamma(\rho/n)),\ \ \ \ (\xi,\rho) \in\,[0,1]\times \reals.\]
It is immediate to check that all functions $f_n$ are in $C^2(\reals)$ and satisfy  \eqref{h11} and \eqref{h12}, so that  the corresponding composition operators $F_n$ satisfiy inequalities \eqref{sa3} and \eqref{dissiFEbis}, for constants $c$ and $\la$ independent of $n$. Namely
 \begin{equation}
\label{sa3n} |D^j F_n(x)|_{E}\leq c\,(1+|x|^{(m-j)^+}_{E}),
\end{equation}
and
\begin{equation}
\label{dissiFEbisn} \le<F_n(x+h)-F_n(x),\d_{h}\r>_{E}\leq
\la\,|h|_E.
\end{equation}
Notice that all $f_n(\xi,\cdot)$ are Lipschitz continuous, uniformly with respect to $\xi \in\,[0,1]$, so that all $F_n$ are Lipschitz continuous in all $L^p(0,1)$ spaces and in $E$.

According to \eqref{ro}, we can easily prove that for any $j=0,1,2$ and $R>0$
\[\lim_{n\to \infty}\sup_{(\xi,\rho) \in\,[0,1]\times[-R,R]}|D^jf_n(\xi,\rho)-D^j f(\xi,\rho)|=0,\]
and then for any $R>0$ and $j=1,2$  

\begin{equation}
\label{6.1}
\le\{\begin{array}{l}
\ds{\lim_{n\to \infty}\sup_{|x|_E\leq R}|F_n(x)-F(x)|_E=0,}\\
\vs
\ds{\lim_{n\to \infty}\sup_{\substack{|x|_E\leq R\\|y_1|_E,\ldots,|y_j|_E\leq R}}|D^j F_n(x)(y_1,\ldots,y_j)-D^j F(x)(y_1,\ldots,y_j)|_E=0.}
\end{array}\r.
\end{equation}

\medskip

We have already seen that the mappings $F_n$ are  Lipschitz-continuous in $H$. The differentiability properties of $F_n$ in $H$ are a more delicate issue. Actually, even if $f_n(\xi,\cdot)$ is assumed to be smooth, $F_n:H\to H$ is only Gateaux differentiable and its Gateaux derivative at $x \in\,H$ along the direction $h \in\,H$ is given by
\[[D F_n(x)h](\xi)=D_\rho f_n(\xi,x(\xi))h(\xi),\ \ \ \ \xi \in\,[0,1].\]
Higher order differentiability is even more delicate, as the higher order derivatives do not exist along any direction in $H$, but only along more regular directions. For example, the second order derivative of $F_n$ exists only along directions in $L^4(0,1)$, and
for any $x \in\,H$ and $h, k \in\,L^4(0,1)$
\[[D^2 F_n(x)(h,k)](\xi)=D^2_\rho f_n(\xi,x(\xi))h(\xi)k(\xi),\ \ \ \xi \in\,[0,1].\]

\section{The solution of \eqref{eq1}} 
\label{sec3}

With the notations introduced in Section \ref{sec2}, equation \eqref{eq1} can be rewritten as the following abstract evolution equation
\begin{equation}
\label{eqabstract}
du(t)=\le[A u(t)+F(u(t))\r]\,dt+G(u(t))dw(t),\ \ \ \ \ u(0)=x.
\end{equation}

\begin{Definition}
An adapted process $u \in\,L^p(\Omega;C([0,T];E))$ is a mild solution for equation \eqref{eqabstract} if
\[u(t)=e^{tA}x+\int_0^t e^{(t-s)A} F(u(s))\,ds+\int_0^t e^{(t-s)A}G(u(s))\,dw(s).\]
\end{Definition}
Let $X=E$ or $X=H$. In what follows, for any $T>0$ and $p\geq 1$ we shall denote by $C^w_{p,T}(X)$ the set of adapted processes in $L^p(\Omega;C([0,T];X))$.  Endowed with the norm
\[\|u\|_{C^w_{p,T}(X)}=\le(\E\sup_{t \in\,[0,T]}|u(t)|_X^p\r)^{\frac 1p},\]$C^w_{p,T}(X)$ is a Banach space.  Furthermore, we shall denote by $L^w_{p,T}(X)$ the Banach space of adapted processes in $C([0,T];L^p(\Omega;X))$, endowed with the norm
\[\|u\|_{L^w_{p,T}(X)}=\sup_{t \in\,[0,T]}\le(\E\,|u(t)|_X^p\r)^{\frac 1p}.\]

In \cite{cerrai} it has been proved that, under Hypothesis \ref{H1},  for any $T>0$ and $p\geq 1$ and for any $x \in\,E$,   equation  \eqref{eqabstract} admits a unique mild solution $u^x$ in $C^w_{p,T}(E)$. Moreover
\begin{equation}
\label{bie} 
\|u^x\|_{C^w_{p,T}(E)}\leq c_{p,T}\le(1+|x|_E\r).
\end{equation}

One of the key steps in the proof of such an existence and uniqueness result, is given in  \cite[Theorem 4.2]{cerrai}, where it proved that the mapping
\[u \in\,C^x_{p,T}(E)\mapsto \le(t\mapsto \Gamma(u)(t):=\int_0^te^{(t-s)A}G(u(s))\,dw(s)\r) \in\,C^x_{p,T}(E)\] is well defined and Lipschitz continuous. By adapting the arguments used in the proof of \cite[Theorem 4.2]{cerrai}, it is also possible to show that
\begin{equation}
\label{bie4}
\E\sup_{s \in\,[0,t]}|\Gamma(u)(s)-\Gamma(v)(s)|_{E}^p \leq c_{p}(t)\int_0^t\E|u(s)-v(s)|_E^p\,ds,\ \ \ \ t\geq 0,
\end{equation}
so that, in particular, there exists $T_p>0$ such that
\begin{equation}
\label{bie9}
\|\Gamma(u)-\Gamma(v)\|_{L^w_{p,T}(E)}\leq \frac 14\,\|u-v\|_{L^w_{p,T}(E)},\ \ \ \ T\leq T_p.
\end{equation}

Now, for any $n \in\,\nat$, we consider the approximating problem
\begin{equation}
\label{eqn}
du(t)=\le[A u(t)+F_n(u(t))\r]\,dt+G(u(t))\,dw(t),\ \ \ \ u(0)=x,
\end{equation}
and we denote by $u^x_n$ its unique mild solution in $C^w_{p,T}(E)$.
As all $F_n$ satisfy \eqref{sa3n} and \eqref{dissiFEbisn},  we have that
\begin{equation}
\label{bien} 
\|u_n^x\|_{C^w_{p,T}(E)}\leq c_{p}(T)\le(1+|x|_E\r),\ \ \ \ \ n \in\,\nat,
\end{equation}
for a function $c_{p}(T)$ independent of $n$.

\medskip

As proved in \cite[Section 3]{cerrai}, the mapping
\[u \in\,C^x_{p,T}(H)\mapsto \le(t\mapsto \Gamma(u)(t):=\int_0^te^{(t-s)A}G(u(s))\,dw(s)\r) \in\,C^x_{p,T}(H)\] is well defined and Lipschitz continuous. Then, as the mapping $F_n:H\to H$ is Lipschitz-continuous, we have that for any $x \in\,H$ and for any $T>0$ and $p\geq 1$ problem \eqref{eqn} admits a unique mild solution $u^x_n \in\,C^x_{p,T}(H)$ such that
\begin{equation}
\label{bienH} 
\|u_n^x\|_{C^w_{p,T}(H)}\leq c_{n,p}(T)\le(1+|x|_H\r).
\end{equation}

\begin{Lemma}
\label{pr1}
Under Hypothesis \ref{H1}, for any $T, R>0$ and $p\geq 1$, we have
\begin{equation}
\label{bie2}
\lim_{n\to \infty}\sup_{|x|_E\leq R}\|u^x_n-u^x\|_{C^w_{p,T}(E)}=0.
\end{equation}
\end{Lemma}

\begin{proof}
For any $n \in\,\nat$ we consider the problem
\begin{equation}
\label{bie3}
d\Gamma_n(t)=A \Gamma_n(t)\,dt+\le[G(u^x_n(t))-G(u^x(t))\r]\,dw(t),\ \ \ \ \Gamma_n(0)=0.
\end{equation}
Problem \eqref{bie3} admits a unique mild solution $\Gamma_n \in\,C^w_{p,T}(E)$ which is given by
\[\Gamma_n(t)=\Gamma(u^x_n)(t)-\Gamma(u^x)(t).\]
Hence, due to \eqref{bie4} we have
\begin{equation}
\label{bie4n}
\E\sup_{t \in\,[0,T]}|\Gamma_n(t)|_{E}^p \leq c_{p}(T)\int_0^T\E|u^x_n(s)-u^x(s)|_E^p\,ds.
\end{equation}
Now, if we define $z_n=u^x_n-u^x-\Gamma_n$, we have that $z_n$ solves the equation
\[\frac{dz_n}{dt}(t)=A z_n(t)+F_n(u^x_n(t))-F(u^x(t)),\ \ \ \ z_n(0)=0,\]
so that, thanks to \eqref{dissiFEbisn},
\[\begin{array}{l}
\ds{\frac {d^-}{dt}|z_n(t)|_E\leq \le<A z_n(t),\d_{z_n(t)}\r>_E+\le<F_n(u^x_n(t))-F_n(u^x(t)),\d_{z_n(t)}\r>_E}\\
\vs
\ds{+\le<F_n(u^x(t))-F(u^x(t)),\d_{z_n(t)}\r>_E\leq \la\,|z_n(t)|_E+|F_n(u^x(t))-F(u^x(t))|_E.}
\end{array}\]
Due to \eqref{bie4n}, by comparison this yields  for any $t \in\,[0,T]$
\[\|u^x_n-u^x\|^p_{C^w_{p,t}(E)}\leq c_{p,T}\,\|F_n(u^x)-F(u^x)\|_{C^w_{p,T}(E)}^p+c_p(T)\int_0^t \|u^x_n-u^x\|^p_{C^w_{p,s}(E)}\,ds,\]
and then, by using  the Gronwall Lemma, we get
\[\|u^x_n-u^x\|^p_{C^w_{p,T}(E)}\leq c_{p,T}\,\|F_n(u^x)-F(u^x)\|_{C^w_{p,T}(E)}^p.\]
This means that \eqref{bie2}  follows once we prove that for any $R>0$
\begin{equation}
\label{bie6}
\lim_{n\to \infty}\sup_{|x|_E\leq R}\|F_n(u^x)-F(u^x)\|_{C^w_{p,T}(E)}^p=0.
\end{equation}

For any $K>0$, we have
\[\begin{array}{l}
\ds{\|F_n(u^x)-F(u^x)\|_{C^w_{p,T}(E)}^p}\\
\vs
\ds{=\E\le(\sup_{t \in\,[0,T]}|F_n(u^x(t))-F(u^x(t))|_E^p;\sup_{t \in\,[0,T]}|u^x_n(t)|_E\vee \sup_{t \in\,[0,T]}|u^x(t)|_E\leq K\r)}\\
\vs
\ds{+\E\le(\sup_{t \in\,[0,T]}|F_n(u^x(t))-F(u^x(t))|_E^p;\sup_{t \in\,[0,T]}|u^x_n(t)|_E\vee \sup_{t \in\,[0,T]}|u^x(t)|_E>K\r),}
\end{array}\]
hence, in view of \eqref{sa3n}, \eqref{bie} and \eqref{bien}, we obtain
\[\begin{array}{l}
\ds{\sup_{|x|_E\leq R}\|F_n(u^x)-F(u^x)\|_{C^w_{p,T}(E)}^p\leq \frac{c_{R,p}(T)}{K}+\sup_{|y|_E\leq K}|F_n(y)-F(y)|_E^p.}
\end{array}\]
Thanks to \eqref{6.1}, as $K$ is arbitrary, this implies \eqref{bie6} and \eqref{bie2} follows.

\end{proof}

\section{The first derivative}
\label{subsec3.1}

For any $x \in\,E$,  $u \in\,L^w_{p,T}(E)$ and $n \in\,\nat$, we define 
\[\Lambda_n(x,u)(t)=e^{tA}x+\int_0^te^{(t-s)A}F_n(u(s))\,ds+\int_0^te^{(t-s)A}G(u(s))\,dw(s),\ \ \ t\geq 0.\]
Clearly, for any $x \in\,E$ the solution $u^x_n$  of problem \eqref{eqn} is the unique fixed point of $\Lambda_n(x,\cdot)$.
The mapping $F_n:E\to E$ is Lipschitz continuous, then due to \eqref{bie9},
there exists $T_p=T_p(n)>0$ such that 
\[\|\Lambda_n(x,u)-\Lambda_n(x,v)\|_{L^w_{p,T}(E)}\leq \frac 12 \|u-v\|_{L^w_{p,T}(E)},\ \ \ T\leq T_p.\]
Therefore, if we show that the contraction mapping $\Lambda_n$ is of class $C^1$, we get that the mapping
\[x \in\,E\mapsto u^x_n \in\,L^w_{p,T}(E),\]
 is differentiable and for any $h \in\,E$
\begin{equation}
\label{bie8}
D_x u^x_n h=D_x \Lambda_n(x,u^x_n) h+D_u \Lambda_n(x,u^x_n) D_x u^x_n h\end{equation}
(for a proof see for example \cite{tesi}).

As $f_n(\xi,\cdot)$ is in $C^2(\reals)$, the mapping $F_n:E\to E$ is twice continuously differentiable, then it is possible to check that the mapping
\[u \in L^w_{p,T}(E)\mapsto \le(t\mapsto \int_0^t e^{(t-s)A}F_n(u(s))\,ds\r) \in\,L^w_{p,T}(E)\]
is twice differentiable.
Analogously, as the mapping $g(\xi,\cdot)$ is in $C^2(\reals)$, by using the stochastic factorization method as in \cite[Theorem 4.2]{cerrai}, it is not difficult to prove that the mapping 
\[u \in L^w_{p,T}(E)\mapsto \le(t\mapsto \int_0^t e^{(t-s)A}G(u(s))\,dw(s)\r) \in\,L^w_{p,T}(E)\]
is twice differentiable. 

Moreover, for any $x \in\,E$ and $u, v \in\,L^w_{p,T}(E)$, we have  
\begin{equation}
\label{bie10}
[D_u \Lambda_n(x,u) v](t)=\int_0^t e^{(t-s)A}F_n^\prime(u(s)) v(s)\,ds+\int_0^t e^{(t-s)A}G^\prime(u(s))v(s)\,dw(s)\ \ \ t\geq 0,
\end{equation}
where, for any $x,y,z \in\,E$ and $\xi \in\,[0,1]$
\[\begin{array}{l}
\ds{[F^\prime_n(x)y](\xi)=D_\rho f_n(\xi,x(\xi))y(\xi),}\\
\vs
\ds{\le[\le(G^\prime(x)y\r)z\r](\xi)=D_\rho g(\xi,x(\xi))y(\xi)z(\xi),}
\end{array}\]
and $D_\rho f_n$ and $D_\rho g$ are the derivatives of $f_n$ and $g$ with respect to the second variable.
Therefore, as, clearly,
\[[D_x \Lambda_n(x,u) h](t)=e^{tA}h,\]
from \eqref{bie8} we have that $\eta^h_n:=D_x u^x_n h$ solves the  linear equation
\begin{equation}
\label{bie12}
d\eta^h_n(t)=\le[A \eta^h_n(t)+F^\prime_n(u^x_n(t))\eta^h_n(t)\r]\,dt+G^\prime(u^x_n(t))\eta^h_n(t)\,dw(t),\ \ \ \eta^h_n(0)=h.
\end{equation}

\begin{Lemma}
\label{bie16}
Under Hypothesis \ref{H1},  for any $T>0$ and $p\geq 1$ the process $u^x_n$ is differentiable with respect to $x \in\,E$ in $L^w_{p,T}(E)$. Moreover, the derivative $D_xu^x_n h=:\eta^h_n$ belongs to  $C^w_{p,T}(E)$ and satisfies
\begin{equation}
\label{bie11}
\|\eta^h_n\|_{C^w_{p,T}(E)}\leq M_{p}\,e^{\omega_p T}\,|h|_E,
\end{equation}
for some  constants $M_p$ and $\omega_p$ independent of $n \in\,\nat$.
\end{Lemma}

\begin{proof} To prove \eqref{bie11} we cannot use the It\^o formula, due to presence of the white noise. Moreover we cannot use the same arguments used for example in  \cite{cerrai} and \cite{tesi}, because of the unboundedness of $f^\prime$ and the presence of the noisy part.
In view of what we have already seen, we have only to prove that \eqref{bie11} holds. To this purpose, the key remark here is that we can assume $h\geq 0$. Actually, in the general case we can decompose $h=h^+-h^-$. As $h^+$ and $h^-$ are non-negative, both $\eta^{h^+}_n$ and $\eta^{h^-}_n$   verify the Lemma and then, since by linearity $\eta^h_n=\eta^{h^+}_n-\eta^{h^-}_n$, we can conclude that the Lemma is true also for $\eta_n^h$.

Let  $\Gamma_n(t)$ be the mild solution of the problem
\[d\Gamma_n(t)=[A-I]\,\Gamma_n(t)\,dt+G^\prime(u^x_n(t))\eta^h_n(t)\,dw(t),\ \ \ \ \Gamma_n(0)=0.\]
Since we are assuming that $D_\rho g(\xi,\cdot)$ is bounded, uniformly with respect to $\xi \in\,[0,1]$, we have that the argument of \cite[Theorem 4.2 and Proposition 4.5]{cerrai} can be adapted to the present situation and 
\begin{equation}
\label{bie14}
\E\sup_{s \in\,[0,t]}|\Gamma_n(s)|^p_E\leq c_{p}\int_0^t \E|\eta_n^h(s)|^p_E\,ds,\ \ \ \ t \in\,[0,T],\end{equation}
for some constant $c_p$ independent of $T>0$.

Next, if we set $z_n=\eta^h_n-\Gamma_n$, we have  that $z_n$ solves the equation
\[\frac{dz_n}{dt}(t)=[A-I]\, z_n(t)+\le[F^\prime_n(u^x_n(t))+I\r]\,\eta_n^h(t),\ \ \ \ z_n(0)=h.\]
Now, since we are assuming that $h\geq 0$ and equation \eqref{bie12} is linear, we have that 
\[\Pro\le(\eta^h_n(t)\geq 0,\ t \in\,[0,T]\r)=1,\]
(see \cite{dmp} for a proof and see also \cite{stabi} for an analogous result for equations with non-Lipschitz coefficients).
Therefore, as $f^\prime_n(\xi,\rho)\leq \la$, for any $(\xi,\rho) \in\,[0,1]\times \reals$, recalling how $\d_x$ is defined in \eqref{bds92} we obtain
\[\begin{array}{l}
\ds{\frac{d^-}{dt}|z_n(t)|_E\leq \le<\le[A-I\r]\,z_n(t),\d_{z_n(t)}\r>_E+\le<\le[F^\prime_n(u^x_n(t))+I\r]\,\eta_n^h(t),\d_{z_n(t)}\r>_E}\\
\vs
\ds{\leq \le(f^\prime_n(\xi_{z_n(t)}, u^x_n(t,\xi_{z_n(t)})+1\r) \eta_n^h(t,\xi_{z_n(t)})\leq (\la+1)^+\,| \eta_n^h(t)|_E.}
\end{array}\]
Thanks to \eqref{bie14}, this implies
\[\begin{array}{l}
\ds{\E\sup_{s \in\,[0,t]}|\eta^h_n(s)|_E^p\leq c_p\,\E\sup_{s \in\,[0,t]}|z_n(s)|^p_E+c_p\,\E\sup_{s \in\,[0,t]}|\Gamma_n(s)|_E^p}\\
\vs
\ds{\leq c_p\,|h|_E^p+c_p\int_0^t\E\sup_{r \in\,[0,s]}|\eta^h_n(r)|_E^p\,ds,}
\end{array}\]
so that from the Gronwall Lemma \eqref{bie11} follows.
\end{proof}

\begin{Lemma}
\label{bie28}
Under Hypothesis \ref{H1},  there exists $\eta^h \in\,C^w_{p.T}(E)$ such that for any $R>0$
\begin{equation}
\label{bie29}
\lim_{n\to \infty}\sup_{x, h \in\,B_R(E)}\|\eta^h_n-\eta^h\|_{C^w_{p.T}(E)}=0.
\end{equation}
Moreover, the limit $\eta^h$ solves the equation
\begin{equation}
\label{bie15}
d\eta^h((t)=\le[A\eta^h(t)+F^\prime(u^x(t))\eta^h(t)\r]\,dt+G^\prime(u^x(t))\eta^h(t)\,dw(t),\ \ \ \eta^h(0)=h,
\end{equation} 
and
\begin{equation}
\label{bie11bis}
\|\eta^h\|_{C^w_{p,T}(E)}\leq M_p\,e^{\omega_p T}\,|h|_E.
\end{equation}

\end{Lemma}

\begin{proof}
For any $n, k \in\,\nat$ we have
\[\begin{array}{l}
\ds{\|\eta^h_{n+k}-\eta^h_n\|_{C^w_{p,T}(E)}^p=\E\le( \sup_{t \in\,[0,T]}|\eta^h_{n+k}(t)-\eta^h_{n}(t)|_E^p; \sup_{t \in\,[0,T]}|u^x(t)|_E\leq n \r)}\\
\vs
\ds{+\E\le( \sup_{t \in\,[0,T]}|\eta^h_{n+k}(t)-\eta^h_{n}(t)|_E^p; \sup_{t \in\,[0,T]}|u^x(t)|_E> n \r).}
\end{array}\]
Since
\[\begin{array}{l}
\ds{\le\{\sup_{t \in\,[0,T]}|u^x(t)|_E\leq n\r\} \subseteq \le\{\eta^h_{n+k}(t)=\eta^h_n(t),\ t \in\,[0,T]\r\},}
\end{array}\]
thanks to \eqref{bie} and \eqref{bie11} we get
\begin{equation}
\label{bie30}
\begin{array}{l}
\ds{\|\eta^h_{n+k}-\eta^h_n\|_{C^w_{p,T}(E)}^{2p}\leq \E\le( \sup_{t \in\,[0,T]}|\eta^h_{n+k}(t)-\eta^h_{n}(t)|_E^{2p}\r)\,
\Pro\le(\sup_{t \in\,[0,T]}|u^x(t)|_E> n \r)}\\
\vs
\ds{\leq \frac{c_p(T)}{n^{2p}}\,|h|^{2p}_E\le(1+|x|^{2p}_E\r),}
\end{array}
\end{equation}
and this implies that  $\{\eta^h_n\}_{n \in\,\nat}$ is a Cauchy sequence in $C^w_{p,T}(E)$.

Let $\eta^h$ be its limit and let $R>0$ and $n \in\,\nat$. For any $m\geq n$ and $x, h \in\,B_R(E)$, due to \eqref{bie30} we have
\[\begin{array}{l}
\ds{\|\eta^h_n-\eta^h\|_{C^w_{p,T}(E)}\leq \|\eta^h_n-\eta^h_m\|_{C^w_{p,T}(E)}+\|\eta^h_m-\eta^h\|_{C^w_{p,T}(E)}}\\
\vs
\ds{\leq \frac{c_p(T,R)}{n}+\|\eta^h_m-\eta^h\|_{C^w_{p,T}(E)}}
\end{array}\]
Therefore, if we fix $\e>0$ and $\bar{m}=m(\e,x,h,\rho,T,p)\geq n$ such that 
\[\|\eta^h_{\bar{m}}-\eta^h\|_{C^w_{p,T}(E)}<\e,\]
due to the arbitrariness of $\e>0$ we get \eqref{bie29}.

 Moreover, as 
\[\eta^h_n(t)=e^{tA}h+\int_0^t e^{(t-s)A}F^\prime_n(u^x_n(s))\eta^h_n(s)\,ds+\int_0^t e^{(t-s)A}G^\prime_n(u^x_n(s))\eta^h_n(s)\,dw(s),\]
and, in addition to \eqref{bie29}, also \eqref{bie2} holds, we can take the limit in both sides and we get that the limit $\eta^h$ is a mild solution of equation \eqref{bie15}.
\end{proof}

\bigskip

\begin{Remark}
\label{R31}
{\em In \cite[Chapter 4]{tesi} the differentiability of the mapping
\begin{equation}
\label{bie40}
x \in\,H\mapsto u^x_n  \in\,L^x_{p,T}(H),\end{equation}
 has been studied, in the case $g(\xi,\rho)=1$. The proof of \cite[Proposition 4.2.1]{tesi} can be adapted to the present situation of a general smooth $g$ and we have that  the mapping in \eqref{bie40} is differentiable and the derivative $D_x u^x_n h$ satisfies equation \eqref{bie12}. Moreover, by proceeding as in \cite[Lemma 4.2.2]{tesi}, it is possible to prove that $D_x u^x_n(t) h \in\,L^p(0,1)$ for any $t>0$, $\Pro$-a.s., and for any $p\geq 2$ and $q\geq 0$
 \begin{equation}
 \label{bie41}
 \sup_{x \in\,H}\E\,|D_xu^x_n (t)h|^q_{L^p(0,1)}\leq c_{p,q,n} (t\wedge 1)^{-\frac{(p-2)q}{4p}}|h|_H^q.
 \end{equation}}
 \end{Remark}
 
Next, we show that we can estimate $\eta^h$ in $H$.

\begin{Lemma}
\label{Lbr24}
Under Hypothesis \ref{H1}, for any $T>0$ and $p\geq 1$ we have
\begin{equation}
\label{br25}
\|\eta^h\|_{C^w_{p,T}(H)}\leq c_{p}(T)\,|h|_H.
\end{equation}
Moreover, for any $p\geq 1$ and $q \in\,[2,+\infty]$ such that $p(q-2)/4q<1$, we have
\begin{equation}
\label{br30}
\E\,|\eta^h(t)|_{L^q(0,1)}^p\leq c_{p,q}(t) t^{-\frac {p(q-2)}{4q}}|h|_H^p,\ \ \ \ t>0.
\end{equation}

If we assume that the constant $\la$ in \eqref{h12} is strictly negative, then for any $p\geq 1$ there exists $\d_p>0$ such that
\begin{equation}
\label{br26}
\E|\eta^h(t)|_H^p\leq c_{p}e^{-\d_p t}\,|h|^p_H,\ \ \ \ t\geq 0,
\end{equation}
and
\begin{equation}
\label{br31}
\E\,|\eta^h(t)|_E^p\leq c_p\,e^{-\d_p t}\, (t\wedge 1)^{-\frac p4}|h|_H^p.
\end{equation}

\end{Lemma}
\begin{proof}
If we denote by $\Gamma^h(t)$ the mild solution of
\[d\gamma(t)=(A+\la)\gamma(t)\,dt+G^\prime(u^x(t))\eta^h(t)\,dw(t),\ \ \ \ \gamma(0)=0,\]
where $\la$ is the constant introduced in \eqref{h12},
we have that $\rho(t):=\eta^h(t)-\Gamma^h(t)$ solves the problem
\[\frac{d\rho(t)}{dt}=(A+\la)\rho(t)+(F^\prime(u^x(t))-\la)\eta^h(t),\ \ \ \ \rho(0)=h.\]
As in the proof of Lemma \ref{bie16}, we decompose $\rho(t)=\rho^+(t)-\rho^-(t)$, where
\[\frac{d\rho^{\pm}(t)}{dt}=(A+\la)\rho^{\pm}(t)+(F^\prime(u^x(t))-\la)\eta^{h^{\pm}}(t),\ \ \ \ \rho(0)=h^{\pm}.\]
As 
\[\Pro\le(\eta^{h^{\pm}}(t)\geq 0,\ \ t \in\,[0,T]\r)=1,\]
and $f^\prime-\la\leq 0$, we have
\begin{equation}
\label{br32}
\rho^{\pm}(t)=e^{t(A+\la)} h^{\pm}+\int_0^t e^{(t-s)(A+\la)}(F^\prime(u^x(s))-\la)\eta^{h^{\pm}}(s)\,ds\leq e^{t(A+\la)} h^{\pm},
\end{equation}
and then 
\[|\rho(t)|_H\leq e^{\la t}|h|_H,\ \ \ \ \Pro-\text{a.s.}\]
Therefore, since for any $q \in\,[2,+\infty]$ and $p\geq 0$
\begin{equation}
\label{br34}
\E\sup_{t \in\,[0,T]}|\Gamma(t)|_{L^q(0,1)}^p\leq c_{p,q}(t)\int_0^t\E\,|\eta^h(s)|_{L^q(0,1)}^p\,ds,
\end{equation}
we can conclude that
\[\E\sup_{s \in\,[0,t]}|\eta^h(s)|_H^p\leq c_p\,e^{\la p t}|h|^p_H+c_{p}(t)\int_0^t\E\,\sup_{r \in\,[0,s]}|\eta^h(r)|_H^p\,ds\]
and \eqref{br25} follows from the Gronwall Lemma.

In order to prove \eqref{br30}, we notice that, due to \eqref{br32},
\[|\eta^h(t)|\leq e^{t(A+\la)}|h|+|\Gamma(t)|,\]
so that, in view of \eqref{br100} and \eqref{br34}, we can conclude that
\[\begin{array}{l}
\ds{ \E\,|\eta^h(t)|^p_{L^q(0,1)}\leq c_{p} |e^{t(A+\la)}|h||^p_{L^q(0,1)}+c_p\,\E|\Gamma(t)|_{L^q(0,1)}^p}\\
\vs
\ds{\leq c_{p,q}(t)t^{-\frac{p(q-2)}{4q}}|h|_H^p+c_{p,q}(t)\int_0^t\E\,|\eta^h(s)|_{L^q(0,1)}^p\,ds.}
\end{array}\]
If $p(q-2)/4q<1$, we can conclude  by a comparison argument.

Finally, \eqref{br26} and \eqref{br31} can be proved by combining together the positivity arguments used above, with the exponential estimates proved in \cite[Examples 4.4 and 4.5]{cerraisiam}.

\end{proof}

\section{The second derivative}
\label{subsec3.2}

Now, we investigate the second order differentiability of $u^x_n$ with respect to $x \in\,E$.
For any process $z \in\,L^w_{p,T}(E)$ and $x \in\,E$ we define
\[[T_n(x)z](t)=\int_0^t e^{(t-s)A}F^\prime_n(u^x_n(s))z(s)\,ds+\int_0^t e^{(t-s)A}G^\prime(u^x_n(s))z(s)\,dw(s),\]
so that  equation \eqref{bie12} can be rewritten as
\[\eta_n^x(t)=e^{tA}h+T_n(x)\,\eta_n^x(t).\]

Due to the boundedness of $D_\rho f_n(\xi,\cdot)$ and $D_\rho g(\xi,\cdot)$, we have that there exists $T_p=T_p(n)>0$ such that for any $x \in\,E$
\[\|T_n(x)\|_{{\cal L}(L^w_{p,T}(E))}\leq \frac 12,\ \ \ \ T\leq T_p,\]
so that
\begin{equation}
\label{bie18}
\eta_n^h=\le[I-T_n(x)\r]^{-1}e^{\cdot A}h.
\end{equation}
Since $f_n$ and $g$ are twice differentiable with bounded derivatives and Lemma \ref{bie16} holds, we have that the mapping 
\[x \in\,E\mapsto T_n(x)z \in\,L^w_{p,T}(E)\]
is differentiable.
Therefore,  we can differentiate both sides in \eqref{bie18} with respect to $x \in\,E$ along the direction $k \in\,E$ and we obtain
\[D_x\eta^h_nk=\le[I-T_n(x)\r]^{-1}D_x\le[T_n(x)\eta^h_n\r] k,\]
so that
\[D_x\eta^h_nk-T_n(x) D_x\eta^h_nk=D_x\le[T_n(x)\eta^h_n\r] k.\]
Now, it is immediate to check that for any $k \in\,E$
\[\begin{array}{l}
\ds{D_x\le[T_n(x)z\r] k\,(t)}\\
\vs
\ds{=\int_0^t e^{(t-s)A}F^{\prime \prime}_n(u^x_n(s))(z(s),\eta^k_n(s))\,ds+\int_0^t e^{(t-s)A}G^{\prime\prime}(u^x_n(s))(z(s),\eta^k_n(s))\,dw(s),}
\end{array}\]
and then $\zeta_n^{h,k}:=D_x\eta^h_n=D^2_x u^x_n(h,k)$ satisfies the equation
\begin{equation}
\label{bie20}
\begin{array}{l}
\ds{d\zeta_n^{h,k}(t)=\le[A\zeta_n^{h,k}(t)+F_n^\prime(u^x_n(t))\zeta_n^{h,k}(t)+F_n^{\prime \prime}(u^x_n(t))(\eta_n^{h}(t),\eta^k_n(t))\r]\,dt}\\
\vs
\ds{+\le[G^\prime(u^x_n(t))\zeta_n^{h,k}(t)+G^{\prime \prime}(u^x_n(t))(\eta_n^{h}(t),\eta^k_n(t))\r]\,dw(t),\ \ \ \ \ \zeta^{h,k}(0)=0.}
\end{array}
\end{equation}

\begin{Lemma}
\label{bie19}
Under Hypothesis \ref{H1},  for any $T>0$ and $p\geq 1$ the process $u^x_n$ is twice differentiable in $L^w_{p,T}(E)$ with respect to $x \in\,E$. Moreover the second derivative $D^2_xu^x_n( h,k)=:\zeta^{h,k}_n$ belongs to  $C^w_{p,T}(E)$ and satisfies
\begin{equation}
\label{bie21}
\|\zeta^{h,k}_n\|_{C^w_{p,T}(E)}\leq c_{p}(T)\le(1+|x|_E^{(m-1)}\r)|h|_E\,|k|_E,
\end{equation}
for some continuous increasing function $c_p(T)$ independent of $n \in\,\nat$.
\end{Lemma}

\begin{proof}
We have already seen that $u^x_n$ is twice differentiable in $L^w_{p,T}(E)$ and $D^2_xu^x_n( h,k)$ satisfies equation \eqref{bie20}. Hence, it only remains to prove estimate \eqref{bie21}. 

As proved in Lemma \ref{bie16}, for any $x\in\,E$ and any $h \in\,L^p(\Omega;E)$ which is ${\cal F}_s$-measurable, the equation
\begin{equation}
\label{bie22}
d\eta(t)=\le[A \eta(t) +F^\prime_n(u^x_n(t))\eta(t)\r]\,dt+G^\prime(u^x_n(t))\eta(t)\,dw(t),\ \ \ \eta(s)=h,
\end{equation}
admits a unique solution $\eta^h_n(s,\cdot) \in\,L^p(\Omega;C([s,T];E))$ such that
\[\E\sup_{t \in\,[s,T]}|\eta^h_n(s,t)|_E^p\leq M_p\,e^{\omega_p(T-s)}\E\,|h|^p_E.\]
Hence, we can associate to equation \eqref{bie22} a stochastic evolution operator $\Phi_n(t,s)$ such that
\[\eta^h_n (s,t)=\Phi_n(t,s)h,\ \ \ \ h \in\,L^p(\Omega;E),\]
and such that
\begin{equation}
\label{bie25}
\E\sup_{r\in\,[s,t]}|\Phi_n(r,s)h|_E^p\leq M_p\,e^{\omega_p(t-s)}\E|h|_E^p,\ \ \ 0\leq s\leq t.
\end{equation}
We claim that $\zeta_n^{h,k}$ can be represented in terms of the operator $\Phi_n(s,t)$ as
\begin{equation}
\label{bie24}
\zeta^{h,k}_n(t)=\Gamma_n^{h,k}(t)+\int_{0}^t\Phi_n(t,s)\Sigma^{h,k}_n(s)\,ds,
\end{equation}
where 
$\Gamma_n^{h,k}$ is the solution of the problem
\begin{equation}
\label{bie23}
d\Gamma(t)=A \Gamma(t)\,dt+\le[G^\prime(u^x_n(t))\Gamma(t)+G^{\prime \prime}(u^x_n(t))(\eta^h_n(t),\eta^k_n(t))\r]\,dw(t),\ \ \ \Gamma(0)=0,\end{equation}
and
\[\Sigma^{h,k}_n(t)=F_n^\prime(u^x_n(t))\Gamma_n^{h,k}(t)+F_n^{\prime \prime}(u^x_n(t))(\eta^h_n(t),\eta^k_n(t)).\]
Clearly, in order to prove \eqref{bie24} we have to show that $\int_{0}^t\Phi_n(t,s)\Sigma^{h,k}_n(s)\,ds$ solves the problem
\[dz(t)=\le[Az(t)+F^\prime_n(u^x_n(t))z(t)+\Sigma^{h,k}_n(t)\r]\,dt+G^\prime(u^x_n(t))z(t)\,dw(t),\ \ \ \ z(0)=0.\]
More in general, we have to prove that for any $\Sigma \in\,C^w_{p,T}(E)$ the mild solution of the problem
\begin{equation}
\label{bie26}
dz(t)=\le[Az(t)+F^\prime_n(u^x_n(t))z(t)+\Sigma(t)\r]\,dt+G^\prime(u^x_n(t))z(t)\,dw(t),\ \ \ \ z(0)=0,
\end{equation}
is given by
\[\hat{z}(t):=\int_0^t \Phi_n(t,s)\Sigma(s)\,ds.
\]
We have
\[\begin{array}{l}
\ds{\int_0^te^{(t-s)A}F^\prime_n(u^x_n(s))\hat{z}(s)\,ds=\int_0^te^{(t-s)A}F^\prime_n(u^x_n(s))\int_0^s\Phi_n(s,r)\Sigma(r)\,dr\,ds}\\
\vs
\ds{=\int_0^t\int_r^te^{(t-s)A}F^\prime_n(u^x_n(s))\Phi_n(s,r)\Sigma(r)\,ds\,dr}
\end{array}
\]
and analogously, by the stochastic Fubini theorem,
\[\begin{array}{l}
\ds{\int_0^te^{(t-s)A}G^\prime(u^x_n(s))\hat{z}(s)\,dw(s)=\int_0^te^{(t-s)A}G^\prime(u^x_n(s))\int_0^s\Phi_n(s,r)\Sigma(r)\,dr\,dw(s)}\\
\vs
\ds{=\int_0^t\int_r^te^{(t-s)A}G^\prime_n(u^x_n(s))\Phi_n(s,r)\Sigma(r)\,dw(s)\,dr}
\end{array}
\]
Now, recalling the definition of $\Phi_n(t,s)\Sigma$, we have
\[\begin{array}{l}
\ds{\int_0^t\le[\int_r^te^{(t-s)A}F^\prime_n(u^x_n(s))\Phi_n(s,r)\Sigma(r)\,ds+\int_r^te^{(t-s)A}G^\prime_n(u^x_n(s))\Phi_n(s,r)\Sigma(r)\,dw(s)\r]\,dr}\\
\vs
\ds{=\int_0^t\le[\Phi_n(t,r)\Sigma(r)-e^{(t-r)A}\Sigma(r)\r]\,dr=\hat{z}(t)-\int_0^te^{(t-r)A}\Sigma(r)\,dr,}
\end{array}\]
so that $\hat{z}$ is the mild solution of equation \eqref{bie26}.

Once we have representation \eqref{bie24} for $\zeta_n^{h,k}$, we can proceed with the proof of estimate \eqref{bie21}. As $\Gamma_n^{h,k}$ solves equation \eqref{bie23}, we have
\[\Gamma_n^{h,k}(t)=\int_0^te^{(t-s)A}\le[G^\prime(u^x_n(s))\Gamma_n^{h,k}(s)+G^{\prime \prime}(u^x_n(s))(\eta^h_n(s),\eta^k_n(s))\r]\,dw(s).\]
Therefore, due to the boundedness of $D_\rho g(\xi,\rho)$ and $D^2_\rho g(\xi,\rho)$, from \eqref{bie11} and \eqref{bie4} we get
\[\begin{array}{l}
\ds{\E\sup_{s \in\,[0,t]}|\Gamma_n^{h,k}(s)|_E^p\leq c_p(T)\int_0^t\E\,|\Gamma^{h,k}_n(s)|_E^p\,ds+c_p(T)\,|h|_E^p\,|k|_E^p,}
\end{array}\]
so that, from the Gronwall Lemma, we can conclude
\begin{equation}
\label{bie27}
\E\sup_{s \in\,[0,t]}|\Gamma_n^{h,k}(s)|_E^p\leq c_p(T)\,|h|_E^p\,|k|_E^p.
\end{equation}
Next, as  \eqref{bien} and \eqref{bie11} hold and as the derivatives of $F_n$ satisfy \eqref{sa3n}, due to \eqref{bie25} we have
\[\begin{array}{l}
\ds{\E\sup_{t \in\,[0,T]}\le|\int_{0}^t\Phi_n(t,s)\Sigma^{h,k}_n(s)\,ds\r|^p_E \leq c_p(T)\int_0^T\E\le|\Sigma^{h,k}_n(s)\r|_E^p\,ds}\\
\vs
\ds{\leq c_p(T)\le(1+|x|_E^{(m-1)p}\r)\|\Gamma^{h,k}_n\|_{C^w_{2p,T}(E)}^{\frac 12}+c_p(T)\le(1+|x|_E^{(m-2)p}\r)|h|_E^p |k|_E^p,}
\end{array}\]
and then, thanks to \eqref{bie27}, we get
\[\E\sup_{t \in\,[0,T]}\le|\int_{0}^t\Phi_n(t,s)\Sigma^{h,k}_n(s)\,ds\r|^p_E \leq c_p(T)\le(1+|x|_E^{(m-1)p}\r)|h|_E^p |k|_E^p.\]
Together with \eqref{bie27}, this implies \eqref{bie21}.
\end{proof}

In view of the previous lemmas, by arguing as in the proof of Lemma \ref{bie28}, we get the following result.

\begin{Lemma}
\label{bie31}
Under Hypothesis \ref{H1},  there exists $\zeta^{h,k} \in\,C^w_{p.T}(E)$ such that for any $R>0$
\begin{equation}
\label{bie 29}
\lim_{n\to \infty}\sup_{x, h, k \in\,B_R(E)}\|\zeta^{h,k}_n-\zeta^{h,k}\|_{C^w_{p.T}(E)}=0.
\end{equation}
Moreover, the limit $\zeta^{h,k}$ solves the equation
\[\begin{array}{l}
\ds{d\zeta(t)=\le[A\zeta(t)+F^\prime(u^x(t))\zeta(t)+F^{\prime \prime}(u^x(t))(\eta^h(t),\eta^k(t))\r]\,dt}\\
\vs
\ds{+\le[G^\prime(u^x(t))\zeta(t)+G^{\prime \prime}(u^x(t))(\eta^h(t),\eta^k(t))\r]\,dw(t),\ \ \ \ \zeta(0)=0.}
\end{array}\]
In particular,
\begin{equation}
\label{bie21bis}
\|\zeta^{h,k}\|_{C^w_{p,T}(E)}\leq c_{p}(T)\le(1+|x|_E^{m-1}\r)|h|_E\,|k|_E.
\end{equation}

\end{Lemma}

As a consequence of Lemmas \ref{pr1}, \ref{bie16}, \ref{bie28}, \ref{bie19} and \ref{bie31}, we have the following fact.

\begin{Theorem} 
\label{bie34}
Under Hypothesis \ref{H1}, the mapping
\[x \in\,E\mapsto u^x \in\,L^w_{p,T}(E)\]
is differentiable and the derivative $D_xu^u h$ along the direction $h \in\,E$ solves the problem
\begin{equation}
\label{bie32}
d\eta((t)=\le[A\eta(t)+F^\prime(u^x(t))\eta(t)\r]\,dt+G^\prime(u^x(t))\eta(t)\,dw(t),\ \ \ \eta(0)=h.
\end{equation}
\end{Theorem}

\begin{proof}
For any $n \in\,\nat$ and $x,h \in\,E$ we have
\[u^{x+h}_n-u^x_n=D_x u^x_n\, h+\int_0^1\int_0^1 D^2_x u^{x+\rho \theta h}_n(h,h)\,d\theta\,d\rho.\]
Then, due to \eqref{bie2}, \eqref{bie29} and \eqref{bie 29}, we can take the limit as $n\to \infty$ and we get
\[u^{x+h}-u^x=\eta^h+\int_0^1\int_0^1 \zeta^{h,h}\,d\theta\,d\rho.\]
The mapping
\[h \in\,E\mapsto \eta^h \in\,L^w_{p,T}(E),\]
is clearly linear and according to \eqref{bie11bis} is bounded. Moreover, according to  \eqref{bie21bis} 
we have
\[\le\|\int_0^1\int_0^1 \zeta^{h,h}\,d\theta\,d\rho\r\|_{L^w_{p,T}(E)}\leq c_{p}(T)\le(1+|x|_E^{m-1}\r)|h|_E^2,\]
and then we can conclude that $u^x$ is differentiable in $L^w_{p,T}(E)$ with respect to $x \in\,E$ and its derivative along the direction $h \in\,E$ solves problem \eqref{bie32}.

\end{proof}

In view of Lemma \ref{Lbr24} and Theorem \ref{bie34}, for any $T>0$, $p\geq 1$ and $x, y \in\,E$ we have
\begin{equation}
\label{br27}
\E\sup_{t \in\,[0,T]}|u^x(t)-u^y(t)|^p_H\leq c_{p}(T)|x-y|_H^p.
\end{equation}
Now, if $x \in\,H$ and $\{x_n\}_{n \in\,\nat}$ is any sequence in $E$, converging to $x$ in $H$, due to \eqref{br27} we have that  $\{u^{x_n}\}_{n \in\,\nat}$ is a Cauchy sequence in $C^w_{p,T}(H)$ and then there exists a limit $u^x \in\,C^w_{p,T}(H)$, only depending on $x$,  such that
\begin{equation}
\label{br28}
\|u^x\|_{C^w_{p,T}(H)}\leq c_{p}(T)\le(1+|x|_H\r).\end{equation}
Such a solution will be called {\em generalized solution}.

\begin{Theorem}
\label{Tbr29}
Under Hypothesis \ref{H1}, for any $x\in\,H$ equation \eqref{eqabstract} admits a unique generalized solution $u^x \in\,C^w_{p,T}(H)$, for any  $T>0$ and  $p\geq 1$. Moreover estimate \eqref{br28} holds.
\end{Theorem}

\bigskip

\begin{Remark}
\label{R32}
{\em Concerning the second order differentiability of mapping \eqref{bie40}, we can adapt again the arguments used in \cite[Theorem 4.2.4]{tesi} to the present situation and thanks to \eqref{bie41} we have that mapping \eqref{bie40}  is twice differentiable with respect to $x \in\,H$ and the derivative along the directions $h, k \in\,H$ satisfies equation \eqref{bie20}. Moreover, 
\begin{equation}
\label{bie42}
\|D^2_x u^x_n(h,k)\|_{C^w_{p,T}(H)}\leq c_{p,n}(T)\le(1+|x|_H\r)|h|_H |k|_H.
\end{equation}}
\end{Remark}

\section{The transition semigroup}

We define the transition semigroup associated with equation \eqref{eqabstract}
as
\[P_t\varphi(x)=\E\,\varphi(u^x(t)),\ \ \ x \in\,E,\ \ \ t\geq 0,\]
for any $\varphi \in\,B_b(E)$. In view of Theorem \ref{bie34}, we have that 
\begin{equation}
\label{bie33}
\varphi \in\,C^1_b(E)\Longrightarrow P_t \varphi \in\,C^1_b(E),\ \ \ t\geq 0,
\end{equation}
 and there exist $M>0$ and $\omega \in\,\reals$ such that
\[
\|P_t\varphi\|_1\leq M e^{\omega t}\|\varphi\|_1, \ \ \ \ t\geq 0.
\]

We would like to stress that, in view of Theorem \ref{Tbr29}, the semigroup $P_t$ can be restricted to $C_b(H)$. Actually, for any $\varphi \in\,B_b(H)$ we can define
\[P^H_t\varphi(x)=\E\varphi(u^x(t)),\ \ \ t\geq 0,\ \ \ x \in\,H,\]
where $u^x(t)$ is the unique generalized solution of \eqref{eqabstract}  in $C^w_{p,T}(H)$ introduced in Theorem \ref{Tbr29}. Notice that if $x \in\,E$ and $\varphi \in\,B_b(H)$, then $P^H_t\varphi(x)=P_t\varphi(x).$

\bigskip

Our first purpose here is to prove that the semigroup $P_t$ has a smoothing effect in $B_b(E)$. Namely, we want to prove that $P_t$ maps $B_b(E)$ into $C^1_b(E)$, for any $t>0$.
To this purpose, we have to assume the following condition on the multiplication coefficient $g$ in front of the noise.

\begin{Hypothesis}
\label{H2}
We have
\begin{equation}
\inf_{(\xi,\rho) \in\,[0,1]\times \reals}|g(\xi,\rho)|=:\beta>0.
\end{equation}
\end{Hypothesis}

First of all, we introduce the transition semigroup $P^n_t$ associated with the approximating equation \eqref{eqn}, by setting
\[P^n_t\varphi(x)=\E\,\varphi(u_n^x(t)),\ \ \ x \in\,E,\ \ \ t\geq 0,\]
for any $\varphi \in\,B_b(E)$. It is important to stress that, according to Lemmas \ref{bie16} and \ref{bie19}, 
\begin{equation}
\label{bie33nbis}
\varphi \in\,C^k(E)\Longrightarrow P_t^n \varphi \in\,C^k_b(E),\ \ \ t\geq 0,\ \ \ k=0,1,2,
\end{equation}
 and 
\[
\|P^n_t\varphi\|_k\leq M\,e^{\omega t}\|\varphi\|_k, \ \ \ \ t\geq 0,\ \ k=0,1,2,
\]
for some constants $M>0$ and $\omega \in\,\reals$, which are independent of $n \in\,\nat$.

Notice that, as equation \eqref{eqn} is  solvable in $H$, we can also consider the restriction of $P_t^n$ to $B_b(H)$. In view of what we have seen in Remarks \ref{R31} and \ref{R32}, we have that 
\begin{equation}
\label{bie33H}
\varphi \in\,C^k_b(H)\Longrightarrow P^n_t \varphi \in\,C^k_b(H),\ \ \ t\geq 0,\ \ \ k=1,2,
\end{equation}
 and there exist constant $M_n >0$ and $\omega_n \in\,\reals$ such that
\begin{equation}
\label{bie34H}
\|P^n_t\varphi\|_k\leq M_n e^{\omega_n t}\|\varphi\|_k,\ \ \ \ t\geq 0,\ \ k=1,2.
\end{equation}
Now, due to Hypothesis \ref{H2}, for any $x, y \in\,H$ we can define
\[[G^{-1}(x)y](\xi)=\frac{y(\xi)}{g(\xi,x(\xi))},\ \ \ \ \xi \in\,[0,1].\]
It is immediate to check that for any $p \in\,[1,+\infty]$
\[G^{-1}:H\to {\cal L}(L^p(0,1),L^p(0,1))\]
and
\[G^{-1}(x)G(x)=G(x)G^{-1}(x),\ \ \ x \in\,H.\]
Therefore, we can adapt the proof of \cite[Proposition 4.4.3 and Theorem 4.4.5]{tesi} to the present situation and we can prove that $P^n_t$ has a smoothing effect. Namely, we have
\[\varphi \in\,B_b(H)\Longrightarrow P^n_t\varphi \in\,C^2_b(H),\ \ \ t>0,\]
and the Bismut-Elworthy-Li formula holds
\begin{equation}
\label{bie44}
\le<h,D(P^n_t\varphi)(x)\r>_H=\frac 1t\,\E\,\varphi(u^x_n(t))\int_0^t\le<G^{-1}(u^x_n(s))D_x u^x_n(s)h,dw(s)\r>_H,\ \ \ t>0,
\end{equation}
for any $\varphi \in\,C_b(H)$ and $x, h \in\,H$.

In view of all these results, by proceeding as in the proof of \cite[Theorem 6.5.1]{tesi}, due to what we have proved in Sections \ref{sec3}, \ref{subsec3.1} and \ref{subsec3.2}  we obtain the following fact.
\begin{Theorem}
\label{bie45}
Under Hypotheses \ref{H1} and \ref{H2}, we have 
\[\varphi \in\,B_b(E)\Longrightarrow P_t\varphi \in\,C^1_b(E),\ \ \ t>0,\]
and
\begin{equation}
\label{bie46bis}
\le<h,D(P_t\varphi)(x)\r>_E=\frac 1t\,\E\,\varphi(u^x(t))\int_0^t\le<G^{-1}(u^x(s))D_x u^x(s)h,dw(s)\r>_H,\ \ \ t>0.\end{equation}
In particular,  for any $\varphi \in\,B_b(E)$
\begin{equation}
\label{bie46}
\sup_{x \in\,E}|D(P_t\varphi)(x)|_{E^\star}\leq c\le(t\wedge 1\r)^{-\frac 12}\|\varphi\|_0,\ \ \ t>0.
\end{equation}
\end{Theorem}

Theorem \ref{bie45} says that  if $\varphi \in\,B_b(E)$, then $P_t\varphi \in\,C^1_b(E)$, for any $t>0$. If we could prove that in fact $P_t\varphi \in\,C^1_b(H)$, then we would have
\[\sum_{i=1}^\infty\le|\le<G(x)e_i,D(P_t\varphi)(x)\r>_H\r|^2=|G^\star(x)D(P_t\varphi)(x)|_H^2<\infty.\]
But in general we have only $P_t\varphi \in\,C^1_b(E)$ and it is not clear in principle whether the sum
\[\sum_{i=1}^\infty\le|\le<G(x)e_i,D(P_t\varphi)(x)\r>_E\r|^2\]
is convergent or not. Next theorem provides a positive answer to this question, which will be of crucial importance for the statement and  the proof of the {\em egalit\'e du carr\'e des champs} and for its application to the Poincar\'e inequality.

\begin{Theorem}
\label{T4.2}
Let $\{e_i\}_{i \in\,\nat}$ be the complete orthonormal basis of $H$ defined in \eqref{basis}.
Then, under Hypotheses \ref{H1} and \ref{H2}, for any $\varphi \in\,C_b(E)$ and $x \in\,E$ we have
\begin{equation}
\label{bie1020}
\sum_{i=1}^\infty\le|\le<G(x)e_i,D(P_t\varphi)(x)\r>_E\r|^2\leq c\,|G(x)|_E^2\|\varphi\|_0^2\le(t\wedge 1\r)^{-1},\ \ \ \ t>0.
\end{equation}
Moreover, if $\varphi \in\,C_b^1(E)$, for any $x \in\,E$ we have
\begin{equation}
\label{bie1020bis}
\sum_{i=1}^\infty\le|\le<G(x)e_i,D(P_t\varphi)(x)\r>_E\r|^2\leq c(t)\,P_t\le(|D\varphi(\cdot)|_{E^\star}^2\r)\!(x)\,|G(x)|_E^2\,t^{-\frac 12},\ \ \ \ t>0,
\end{equation}
for some continuous  increasing function. If we also assume that the constant $\la$ in \eqref{h12} is strictly negative, then there exists $\d>0$ such that
\begin{equation}
\label{bie1020tris}
\sum_{i=1}^\infty\le|\le<G(x)e_i,D(P_t\varphi)(x)\r>_E\r|^2\leq c\,e^{-\d t}\,P_t\le(|D\varphi(\cdot)|_{E^\star}^2\r)\!(x)\,|G(x)|_E^2\,t^{-\frac 12},\ \ \ \ t>0.
\end{equation}
 
\end{Theorem}

\begin{proof}
Assume $\varphi \in\,C_b(E)$ and $x, h \in\,E$. According to \eqref{br25} and \eqref{bie46bis}, for any $t \in\,(0,1]$ we have
\[\begin{array}{l}
\ds{\le|\le<h,D(P_t\varphi)(x)\r>_E\r|=\frac 1t\le|\E\,\varphi(u^x(t))\int_0^t\le<G^{-1}(u^x(s))D_x u^x(s)h,dw(s)\r>_H\r|}\\
\vs
\ds{\leq \frac{\|\varphi\|_0}t\le(\int_0^t\E\,|G^{-1}(u^x(s))D_x u^x(s)h|_H^2\,ds\r)^{\frac 12}\leq  \frac{c\,\|\varphi\|_0}t\le(\int_0^t c(s)\,ds\r)^{\frac 12}|h|_H\leq c\,\|\varphi\|_0\,t^{-\frac 12}|h|_H.}
\end{array}\]
Due to the semigroup law, it follows that for any $t>0$
\begin{equation}
\label{br34poi}
\le|\le<G(x)h,D(P_t\varphi)(x)\r>_E\r|\leq c\,\|\varphi\|_0\,(t\wedge 1)^{-\frac 12}|G(x)|_E\,|h|_H.
\end{equation}
This implies in particular that for any $t>0$ and $x \in\,E$ there exists $\Lambda_\varphi(t,x) \in\,H$ such that
\[\le<G(x)h,D(P_t\varphi)(x)\r>_E=\le<\Lambda_\varphi(t,x),h\r>_H,\ \ \ h \in\,E.\]
Therefore, in view of \eqref{br34poi}
\[\begin{array}{l}
\ds{\sum_{i=1}^\infty\le|\le<G(x)e_i,D(P_t\varphi)(x)\r>_E\r|^2=\sum_{i=1}^\infty\le|\le<\Lambda_\varphi(t,x),
e_i\r>_H\r|^2}\\
\vs
\ds{=|\Lambda_\varphi(t,x)|_H^2\leq c\,\|\varphi\|^2_0\,(t\wedge 1)^{-1}|G(x)|^2_E,}
\end{array}\]
and \eqref{bie1020} holds.

Next, in order to prove \eqref{bie1020bis}, we notice that if $\varphi \in\,C^1_b(E)$, then
\[\le<G(x)h,D(P_t\varphi)(x)\r>_E=\E\le<D u^x(t)G(x)h,D\varphi(u^x(t))\r>_E.\]
According to \eqref{br30}, with $p=2$ and $q=+\infty$, for any $t>0$ we have
\[\begin{array}{l}
\ds{\le|\le<G(x)h,D(P_t\varphi)(x)\r>_E\r|^2\leq \E\,|D\varphi(u^x(t))|^2_{E^\star}\E\,|D u^x(t)G(x)h|_E^2}\\
\vs
\ds{
\leq P_t(|D\varphi(\cdot)|_{E^\star}^2)(x)c_{2,\infty}(t) t^{-\frac 12}|G(x)|_E^2|h|_H^2.}
\end{array}\]
As above, this implies that for any $t>0$ and $x \in\,E$ there exists $\hat{\Lambda}_\varphi(t,x) \in\,H$
such that
\[\le<G(x)h,D(P_t\varphi)(x)\r>_E=\le<\hat{\Lambda}_\varphi(t,x),h\r>_H\] and as above we can conclude that
\eqref{bie1020bis} holds.

Finally, in order to get \eqref{bie1020tris}, we have to proceed exactly in the same way, by using \eqref{br31} instead of  \eqref{br30}.

\end{proof}

 \section{Kolmogorov operator}

We  define the Komogorov operator $\mathcal K$ in $C_b(E)$ associated with $P_t$, by proceeding as in \cite{cerrai94} and \cite{tesi}. The operator $\mathcal K$  is defined through its resolvent by
\begin{equation}
\label{e5.1}
(\lambda-\mathcal K)^{-1}\varphi(x)=\int_0^{+\infty}e^{-\la t} P_t\varphi(x)\,dt,\ \ \ x \in\,E,
\end{equation}
for all $\lambda>0$ and $\varphi\in C_b(E)$,  see also \cite{Priola}.

We notice that, by Theorem \ref{bie45}, we have 
\begin{equation}
\label{e5.2}
D(\mathcal K)\subset C^1_b(E),
\end{equation} 
where $D(\mathcal K)$  is the domain of $\mathcal K$. In fact, this stronger property holds.
\begin{Theorem}
\label{bie2033}
Let $\{e_i\}_{i \in\,\nat}$ be the complete orthonormal basis of $H$ defined in \eqref{basis}.
Then, under Hypotheses \ref{H1} and \ref{H2}, for any $\varphi \in\,D({\cal K})$ and $x \in\,E$ we have
\begin{equation}
\label{bie10200}
\sum_{i=1}^\infty\le|\le<G(x)e_i,D\varphi(x)\r>_E\r|^2\leq c\,|G(x)|_E^2\le(\|\varphi\|_0^2+\|{\cal K}\varphi\|_0^2\r).
\end{equation}

\end{Theorem}

\begin{proof}
Due to the H\"older inequality, for any $\e \in\,(0,1)$ and $\psi \in\,C_b(E)$ we have
\[\begin{array}{l}
\ds{\le|\le<G(x)e_i,D((1-{\cal K})^{-1}\psi)(x)\r>_E\r|^2\leq \int_0^\infty e^{-t}(t\wedge 1)^{-(1-\e)}\,dt}\\
\vs
\ds{\times \int_0^\infty e^{- t}(t\wedge 1)^{1-\e}\le|\le<G(x)e_i,D(P_t\psi)(x)\r>_E\r|^2\,dt}
\end{array}\]
and then, according to \eqref{bie1020}, we get
\[\begin{array}{l}
\ds{\sum_{i=1}^\infty \le|\le<G(x)e_i,D((1-{\cal K})^{-1}\psi)(x)\r>_E\r|^2\leq c_\e
\int_0^\infty e^{- t}(t\wedge 1)^{-\e}\,dt\,\,|G(x)|_E^2\|\psi\|_0^2.}
\end{array}\]
Therefore, if we take $\psi=(1-{\cal K})\varphi$,  we get \eqref{bie10200}.
\end{proof}

Our  goal   is to prove the following result
\begin{Theorem}
\label{t5.1}
Assume Hypotheses \ref{H1} and \ref{H2}. Then,
for any $\varphi\in D(\mathcal K)$ we have $\varphi^2\in D(\mathcal K)$ and the following identity holds
 \begin{equation}
 \label{e5.3}
 \mathcal K\varphi^2=2\varphi\, \mathcal K\varphi+\sum_{i=1}^\infty |\le< G(\cdot)e_i,D\varphi\r>_E|^2.
 \end{equation} 
\end{Theorem} 
In order to prove  identity \eqref{e5.3}, we need suitable approximations  of problem  \eqref{eq1} besides  \eqref{eqn}.
For any $m\in\nat$, we denote by $u^x_{n,m}$ the unique   mild solution in $C^w_{p,T}(E)$ of the problem
\[
du(t)=\le[A u(t)+F_n(u(t))\r]\,dt+G(u(t))\,P_mdw(t),\ \ \ \ u(0)=x,
\]
where $P_mx=\sum_{i=1}^m\langle x,e_k\rangle e_k$, $x\in H$.
Moreover for any $k\in\nat$ we denote by $u^x_{n,m,k}$ the unique   solution in $C^w_{p,T}(E)$ of the problem
\begin{equation}
\label{e5.5}
du(t)=\le[A_k u(t)+F_n(u(t))\r]\,dt+G(u(t))\,P_mdw(t),\ \ \ \ u(0)=x,
\end{equation}
where $A_k=kA(k-A)^{-1}$ are the Yosida approximations of $A$.
The following result is straightforward.
\begin{Lemma}
\label{l5.2}
Under Hypotheses \ref{H1} and \ref{H2},
for any $x \in\,E$ and $T>0$ we have
\[
\lim_{m\to \infty}|u^x_{n,m}(t)-u_n^x(t)|_E=0,\quad\mbox{\it uniformly on}\;[0,T].
\]
Moreover for any  $x \in\,E,\;m\in\mathbb N$ and $T>0$ we have
 \[
\lim_{k\to \infty}|u^x_{n,m,k}(t)-u_{n,m}^x(t)|_E=0,\quad\mbox{\it uniformly on}\;[0,T].
\]
\end{Lemma}
Let us introduce the approximating Kolmogorov operators.  If  $\varphi \in C_b(E)$ and $\lambda>0$, they are  defined as above throughout their resolvents 
\[
(\lambda-\mathcal K_n)^{-1}\varphi(x)=\int_0^{+\infty}e^{-\la t}\,\E \varphi(u_n^x(t))dt,
\]
\[
(\lambda-\mathcal K_{n,m})^{-1}\varphi(x)=\int_0^{+\infty}e^{-\la t}\, \E \varphi(u_{n,m}^x(t))dt,
\]
and
\[
(\lambda-\mathcal K_{n,m,k})^{-1}\varphi(x)=\int_0^{+\infty}e^{-\la t}\, \E \varphi(u_{n,m,k}^x(t))dt,
\]
for any $\varphi \in\,C_b(E)$ and $x \in\,E$.
From Lemmas \ref{pr1} and \ref{l5.2}, we get the following approximation results.
\begin{Lemma}
\label{l5.3}
Assume Hypotheses \ref{H1} and \ref{H2}. Then, 
for any $\lambda>0$ and  $x \in\,E$, we have
\[
\lim_{n\to \infty}|(\lambda-\mathcal K_{n})^{-1}\varphi(x)-(\lambda-\mathcal K)^{-1}\varphi(x)|_{E}=0.
\]
If moreover $m\in\nat$,
\[
\lim_{m\to \infty}|(\lambda-\mathcal K_{n,m})^{-1}\varphi(x)-(\lambda-\mathcal K_n)^{-1}\varphi(x)|_{E}=0.
\]
If finally $k\in\nat$
 we have
 \[
\lim_{k\to \infty}|(\lambda-\mathcal K_{n,m,k})^{-1}\varphi(x)-(\lambda-\mathcal K_{n,m})^{-1}\varphi(x)|_{E}=0.
\]
\end{Lemma}

\begin{Lemma}
\label{l5.4}
Assume Hypotheses \ref{H1} and \ref{H2}. Then, 
for any $n,m,k\in\nat$ we have  $C^2_b(E)\subset D(\mathcal K_{n,m,k})$ and for any $\varphi\in C^2_b(E)$ we have
 \begin{equation}
 \label{e5.14}
 \mathcal K_{n,m,k}\varphi^2=2\varphi\, \mathcal K_{n,m,k}\varphi+\sum_{i=1}^m|\le<G(\cdot)e_i,D\varphi \r>_E|^2.
 \end{equation} 
\end{Lemma}
\begin{proof}
Since the stochastic equation \eqref{e5.5} has regular coefficients and a finite dimensional noise term, the conclusion follows from It\^o's formula in the Banach space $E$ (see Appendix A). 
\end{proof}
\begin{Corollary}
\label{c5.5}
Let $\varphi_{n,m,k}=(\lambda- \mathcal K_{n,m,k})^{-1}\psi$, for $n,m,k\in\nat$ and $\psi\in C^2_b(E)$. Then, under Hypotheses \ref{H1} and \ref{H2}, the following identity  holds
 \begin{equation}
 \label{e5.15}
\varphi_{n,m,k}^2=(2\lambda- \mathcal K_{n,m,k})^{-1}(2\varphi_{n,m,k}\,\psi+\sum_{i=1}^m\le|\le< G(\cdot)e_i,D\varphi_{n,m,k}\r>_E\r|^2).
 \end{equation} 
\end{Corollary}
\begin{proof}
As
 \begin{equation}
 \label{e5.16}
\lambda\varphi_{n,m,k}- \mathcal K_{n,m,k}\varphi_{n,m,k}=\psi,
 \end{equation} 
since $\psi\in C^2_b(E)$  we have $\varphi_{n,m,k}\in C^2_b(E)$. Now, multiplying \eqref{e5.16} by $\varphi_{n,m,k}$ and taking into account  \eqref{e5.14}, we get
$$
\lambda\varphi^2_{n,m,k}-\frac12\;\mathcal K_{n,m,k}(\varphi_{n,m,k}^2)-\frac 12\sum_{i=1}^m\le|\le<G(\cdot)e_i,D\varphi_{n,m,k} \r>_E\r|^2
=\psi\varphi_{n,m,k}
$$
and the conclusion follows.
\end{proof}
\begin{Lemma}
\label{l5.6}
Let   $\varphi=(\lambda- \mathcal K)^{-1}\psi$, for $\psi\in C^2_b(E)$ and $\lambda>0$. Then, under Hypotheses \ref{H1} and \ref{H2},  the following identity  holds
 \begin{equation}
 \label{e5.17}
\varphi^2=(2\lambda- \mathcal K)^{-1}(2\varphi\,\psi+\sum_{i=1}^\infty \le|\le<G(\cdot)e_i,D\varphi\r>_E\r|^2).
 \end{equation} 
 Consequently, $\varphi^2\in D(\mathcal K)$ and \eqref{e5.3} holds.

\end{Lemma}
\begin{proof}
The conclusion follows from Theorem \ref{bie2033}, 
Lemma \ref{l5.3} and Corollary \ref{c5.5}, by letting $n,m,k\to \infty$.
\end{proof}
We are now in the position to prove Theorem \ref{t5.1}.

\begin{proof}  
Let $\varphi\in D(\mathcal K)$, $\lambda>0$ and  $\psi=\lambda\varphi-\mathcal K\varphi $. If  we assume that $\psi \in C^2_b(E)$, then, due to Lemma \ref{l5.6}, we know   that \eqref{e5.17} holds.
Now assume $\psi \in\,C_b(E)$.  It is well known that we cannot find an uniform approximation of $\psi$, because $C^2_b(E)$  is not dense in $C_b(E)$. Thus, we define
$$
R_t\psi(x)=\int_H\psi(e^{tA}x+y)N_{Q_t}(dy),
$$
where $N_{Q_t}$ is the Gaussian measure in $H$ with mean $0$ and covariance $Q_t=-\frac12\;A^{-1}(1-e^{2tA})$ for $t\ge 0.$ As $N_{Q_t}$  is the law of the solution of the linear equation 
\[du(t)=A u(t)\,dt+dw(t),\ \ \ u(0)=0\]
which takes values in $E$ and
$e^{tA}x \in\,E$, for any $x \in\,H$ and $t>0$, we have that 
$R_t\psi \in\,B_b(H)$. 
Moreover, as proved in \cite{DPZ3},  we have that  for each $t>0$, $R_t \psi$ belongs to $C^\infty_b(H)$ and consequently to $C^\infty_b(E)$.

Now let $\varphi_t=(\lambda-\mathcal K)^{-1}R_t\psi$. Since $R_t\psi \in C^2_b(E)$,    we have
by \eqref{e5.17}
\begin{equation}
 \label{e5.18}
\varphi_t^2=(2\lambda- \mathcal K)^{-1}(2\varphi_t\,R_t\psi+\sum_{i=1}^\infty \le|\le<G(\cdot)e_i,D\varphi_t\r>_E\r|^2).
 \end{equation} 
 Therefore, the conclusion follows letting $t\to 0$.
Actually, if for any $x \in\,E$ we have
\[\lim_{t\to 0}h_t(x)=h(x),\ \ \ \ \sup_{t \in\,[0,1]}\sup_{x \in\,E}|h_t(x)|<\infty,\]
then it is immediate to check that 
\[\lim_{t\to 0}(\la-{\cal K})^{-1}h_t(x)= (\la-{\cal K})^{-1}h(x),\ \ \ \ x \in\,E.\]
Therefore, as for any $x \in\,E$
\[\lim_{t\to 0}\varphi_t^2(x)=\varphi^2(x),\ \ \ \ \lim_{t\to 0}\varphi_t(x)R_t\psi(x)=\varphi^2(x)\psi(x),\]
we get \eqref{e5.17} by taking the limit as $t\downarrow 0$ in both sides of \eqref{e5.18} if we show that
\begin{equation}
\label{br3332}
\lim_{t\to 0}\sum_{i=1}^\infty\le|\le<G(x)e_i,D\varphi_t(x)\r>_E\r|^2=\sum_{i=1}^\infty\le|\le<G(x)e_i,D\varphi(x)\r>_E\r|^2.
\end{equation}
Now, since
\[\lim_{t\to 0}R_t\psi(x)=\psi(x),\ \ \ \ \|R_t\psi\|_0\leq \|\psi\|_0,\]
according to \eqref{bie46bis} we have
\begin{equation}
\label{br333}
\lim_{t\to 0}\le<G(x)e_i,D\varphi_t(x)\r>_E=\le<G(x)e_i,D\varphi(x)\r>_E,
\end{equation}
for any $i \in\,\nat$. By proceeding as in the proof of Theorem \ref{T4.2}, we see that for any $\varphi \in\,C_b(E)$
\[\le|\le<G(x)h,D(P_t\varphi)(x)\r>_E\r|\leq \|\varphi\|_0\frac 1t\le|\E\int_0^t\le<G^{-1}(u^x(s))Du^x(s)G(x)h,dw(s)\r>_H\r|.\]
Now, as
\[\frac 1t\le|\E\int_0^t\le<G^{-1}(u^x(s))Du^x(s)G(x)h,dw(s)\r>_H\r|\leq ct^{-\frac 12}|G(x)|_E|h|_H,\]
we have that there exists $\Lambda(t,x) \in\,H$ such that
\[\frac 1t\le|\E\int_0^t\le<G^{-1}(u^x(s))Du^x(s)G(x)h,dw(s)\r>_H\r|=\le<\Lambda(t,x),h\r>_H\] and
\begin{equation}
\label{br3331}
|\Lambda(t,x)|_H\leq c\,t^{-\frac 12}|G(x)|_E.
\end{equation}
By arguing as in the proof of Theorem \ref{bie2033}, with $\e=1/2$, this implies that
\[\begin{array}{l}
\ds{\le|\le<G(x)e_i,D\varphi_t(x)\r>_E\r|^2\leq c\int_0^\infty e^{-\la s}(s\wedge 1)^{\frac 12}\le|\le<G(x)e_i,D(P_s(R_t\psi))(x)\r>_E\r|^2\,ds}\\
\vs
\ds{\leq c\,\|R_t\varphi\|_0\int_0^\infty e^{-\la s}(s\wedge 1)^{\frac 12}\le|\le<\Lambda(s,x),e_i\r>_H\r|^2\,ds\leq c\,\|\varphi\|_0\int_0^\infty e^{-\la s}(s\wedge 1)^{\frac 12}\le|\le<\Lambda(s,x),e_i\r>_H\r|^2\,ds.}
\end{array}\]

Therefore, as due to \eqref{br3331}
\[\sum_{i=1}^\infty\int_0^\infty e^{-\la s}(s\wedge 1)^{\frac 12}\le|\le<\Lambda(s,x),e_i\r>_H\r|^2\,ds<\infty,\]
and \eqref{br333} holds, from the Fatou's lemma we get \eqref{br3332} and \eqref{e5.3} follows for a general 
$\varphi \in\,D({\cal K})$.

  \end{proof}

 \section{Invariant measures}

In \cite{cerrai} it has been proved that there exists an invariant probability measure $\mu$ on $(E,\mathcal B(E))$ for the semigroup $P_t$,
 In particular, if $\varphi \in\,D({\cal K})$ we have
  \begin{equation}
 \label{e5.20}
 \int_E \mathcal K\varphi\,d\mu=0.
 \end{equation}
 
 From now on we shall assume that the following condition is satisfied.
 \begin{Hypothesis}
 \label{H3}
 There exists $\a>0$ such that
$A+F^\prime \leq -\a\,I$,
 and $g(\xi,\rho)$ is uniformly bounded on $[0,1]\times \reals$.
 \end{Hypothesis}
 
  In \cite[Proposition 4.1]{cerraisiam}, we have proved that under Hypothesis \ref{H3} there exists $\d>0$ such that for any $p\geq 1$ and $x \in\,E$
 \begin{equation}
 \label{bie2035}
 \E\,|u^x(t)|_E^p\leq c_p\le(1+e^{-\d p t}|x|_E^p\r),\ \ \ t\geq 0.
 \end{equation}
 As a consequence of this, we have that for any $p\geq 1$
 \begin{equation}
 \label{bie10220}
 \int_E |x|_E^p\,\mu(dx)<\infty.
 \end{equation}
 Actually, due to the invariance of $\mu$, for any $t\geq 0$ it holds
 \[\begin{array}{l}
 \ds{\int_E|x|_E^p\,\mu(dx)=\int_E\E\,|u^x(t)|_E^p\,\mu(dx)\leq c_p\le(1+e^{-\d p t}\int_E|x|_E^p\,\mu(dx)\r).}
 \end{array}\]
 Therefore, if we choose $t_0$ such that $c_p e^{-\d p t_0}<1/2$, we have that \eqref{bie10220} follows.
 
\begin{Remark}
{\em   In order to have \eqref{bie2035} it is not necessary to assume that $g$ is uniformly bounded. Actually, \eqref{h13} is what we need to prove \eqref{bie2035}. In Hypothesis \ref{H3} we are assuming that $g$ is bounded in view of the proof of the Poincar\'e inequality, where we need an estimate uniform with respect to $x \in\,E$.

}
\end{Remark}

Now, as $\mu$ is invariant, it is well known that $P_t$ can be uniquely extended to a semigroup of contractions on $L^2(E,\mu)$ which we shall still denote by $P_t$, whereas we shall denote by $\mathcal K_2$ its infinitesimal generator. 
 \begin{Lemma}
 \label{l5.8}
 Assume Hypotheses \ref{H1}, \ref{H2} and \ref{H3}. Then,
 $ D(\mathcal K)$  is a core for $\mathcal K_2.$ 
 \end{Lemma}
 \begin{proof}
   Let $\psi:=\varphi-\mathcal K_2\varphi$, for  $\varphi\in  D(\mathcal K_2)$ and $\lambda>0$.
  Since $C_b(E)$ is dense in $L^2(E,\mu)$, there exists a sequence $(\psi_n)\subset C_b(E)$ convergent to $\psi$ in $L^2(E,\mu)$. If we set $\varphi_n:=(\lambda-\mathcal K_2)^{-1}\psi_n$ then $\varphi_n\in D(\mathcal K)$ and 
 $$
 \varphi_n\to \varphi,\quad  \mathcal K_2 \varphi_n\to \mathcal K\varphi\quad\mbox{\rm in}\;L^2(E,\mu),
 $$
which shows that  $ D(\mathcal K)$  is a core for ${\cal K}_2.$ 
\end{proof}
  
 \subsection{Consequences of the  ``Egalit\'e du carr\'e des champs''}

 Our first result is the so called {\em Egalit\'e du carr\'e des champs} (see \cite{BE}).
 \begin{Proposition}
 \label{p5.7}
 Assume  that Hypotheses \ref{H1}, \ref{H2} and \ref{H3} hold. Then for any $\varphi\in D(\mathcal K)$ we have
  \begin{equation}
  \label{e5.21}
  \int_E\mathcal K\varphi(x)\;\varphi(x)\;d\mu(x)=-\frac12\;\int_E\sum_{i=1}^\infty \le|\le<G(x)e_i,D\varphi(x)\r>_E\r|^2d\mu(x).
  \end{equation} 
 \end{Proposition} 
 {\bf Proof}.  Let $\varphi\in D(\mathcal K)$.  Then, by Theorem \ref{t5.1}, $\varphi^2\in D(\mathcal K)$ and  identity \eqref{e5.3} holds. According to \eqref{bie10200} and \eqref{bie10220}, we can
integrate both sides of \eqref{e5.3} with respect to $\mu$
and taking into account that, in view of \eqref{e5.20}, $ \int_E\mathcal K(\varphi^2)d\mu=0,$  we get the conclusion. \hfill$\Box$\medskip

  Let us show a similar  identity for the semigroup $P_t$.
 \begin{Proposition}
 \label{p6.4}
Let $\varphi\in C^1_b(E)$ and set $v(t,x)=P_t\varphi(x)$.  Then, under Hypotheses \ref{H1}, \ref{H2} and \ref{H3}, we have
\[v\in L^\infty(0,T;L^2(E,\mu)),\ \ \ \  \sum_{i=1}^\infty\le|\le<G(\cdot)e_i,D_xv\r>_E\r|^2 \in\, L^1(0,T;L^1(E,\mu)),\]
for any $T>0$. Moreover
  \begin{equation}
  \label{e6.5}
  \int_E(P_t\varphi)^2\mu(dx)+\int_0^tds\int_E \sum_{i=1}^\infty\le|\le<G(x)e_i,D(P_s\varphi)(x)\r>_E\r|^2\,\mu(dx)=\int_H\varphi^2(x)\mu(dx).
  \end{equation} 

 \end{Proposition}

 \begin{proof}
If we assume that $\varphi \in\,D({\cal K})$, we have $P_t\varphi \in\,D({\cal K})$ and  ${\cal K}P_t\varphi=P_t {\cal K}\varphi$
(for a proof see \cite[Lemma B.2.1]{tesi}).
According to \eqref{bie10200}, this yields 
 \[\begin{array}{l}
 \ds{\sum_{i=1}^\infty\le|\le<G(x)e_i,D_xv(t,x)\r>_E\r|^2\leq c\,|G(x)|_E^2\le(\|P_t\varphi\|_0^2+\|{\cal K}P_t\varphi\|_0^2\r)}\\
 \vs
 \ds{\leq c\,|G(x)|_E^2\le(\|\varphi\|_0^2+\|{\cal K}\varphi\|_0^2\r),}
 \end{array}\]
 so that 
 \begin{equation}
 \label{br5}
 \sum_{i=1}^\infty\le|\le<G(\cdot)e_i,D_xv(t,\cdot)\r>_E\r|^2 \in\,L^1(0,T;L^1(E,\mu)),\end{equation}
 for any $T>0$.
Now,
  as $D_tv(t,x)=\mathcal K v(t,x)$  (see \cite[Proposition B.2.2]{tesi}), multiplying both sides by $v(t,x)$ and integrating over $E$ with respect to $\mu$, due to  \eqref{e5.21} we get 
 $$
 \begin{array}{l}
\ds \frac12\;\frac{d}{dt}\;\int_Ev^2(t,x)\mu(dx)=\int_E\mathcal Kv(t,x)\,v(t,x)\mu(dx)=-\frac12\;\int_E\sum_{i=1}^\infty\le|\le<G(x)e_i,D_xv(t,x)\r>_E\r|^2\mu(dx).
 \end{array}
 $$
Thus, integrating with respect to $t$, \eqref{e6.5} follows when $\varphi \in\,D({\cal K})$.

Now, assume $\varphi \in\,C^1_b(E)$. Clearly, the mapping  $(t,x)\mapsto P_t\varphi(x)$ is in  $L^\infty(0,T;L^2(E,\mu))$. Moreover, according to \eqref{bie1020bis} 
\begin{equation}
\label{br0}
\begin{array}{l}
\ds{\sum_{i=1}^\infty\le|\le<G(x)e_i,D(P_t\varphi)(x)\r>_E\r|^2\leq c(t)\,P_t(\le|D\varphi(\cdot)\r|_{E^\star}^2)(x)|G(x)|_E^2t^{-\frac 12}}\\
\vs
\ds{\leq c(t)\,\sup_{x \in\,E}|D\varphi(x)|_{E^\star}|G(x)|_E^2\,t^{-\frac 12}.}
\end{array}
\end{equation}
and then \eqref{br5} holds.
Next, for any $n \in\,\nat$ we define
$\varphi_n:=n(n-{\cal K})^{-1}\varphi$. Clearly, $\varphi_n \in\,D({\cal K})$ and for $x \in\,E$
\begin{equation}
\label{br1}
\lim_{n\to \infty}\varphi_n(x)=\varphi(x),\ \ \ \|\varphi_n\|_0\leq \|\varphi\|_0,\ \ \ n \in\,\nat.\end{equation}
Moreover, thanks to \eqref{bie11bis}, we have
\begin{equation}
\label{br2}\lim_{n\to \infty}|D\varphi_n(x)-D\varphi(x)|_{E^\star}=0,\ \ \ \ \sup_{x \in\,E}|D\varphi_n(x)|_{E^\star}\leq \sup_{x \in\,E}|D\varphi(x)|_{E^\star},\ \ \ n \in\,\nat.\end{equation}
As \eqref{e6.5} holds for $\varphi \in\,D({\cal K})$, if we set $v_n(t,x)=P_t\varphi_n(x)$, we have for each $n \in\,\nat$ 
\[  \int_Ev_n^2(t,x)\mu(dx)+\int_0^tds\int_E \sum_{i=1}^\infty\le|\le<G(x)e_i,D_xv_n(s,x)\r>_E\r|^2\,\mu(dx)=\int_H\varphi_n^2(x)\mu(dx).\] 
Due to \eqref{br0}, \eqref{br1} and \eqref{br2}, by arguing as in the proof of Lemma \ref{l5.6}, we can take the limit in both sides above, as $n\to \infty$, and we get \eqref{e6.5} for $\varphi \in\,C^1_b(E)$.

  \end{proof}

\subsection{The Sobolev space $W^{1,2}(E,\mu)$} 

We are going to show that the derivative operator $D$ is closable in $C^1_b(E)$, so that we can introduce the Sobolev space $W^{1,2}(E,\mu)$.

\begin{Proposition}
\label{p6.5} 
Assume Hypotheses \ref{H1}, \ref{H2} and \ref{H3}. Then, the derivative operator
 $$
D:C^1_b(E)\to L^2(E,\mu;E^\star),\quad \varphi\mapsto D\varphi,
$$
is closable.

\end{Proposition}
\begin{proof}
Let $(\varphi_n)\subset C^1_b(E)$  such that
$$
\varphi_n\to 0\quad\mbox{\rm in}\;  L^2(E,\mu)\quad
D\varphi_n\to F\quad\mbox{\rm in }\;  L^2(E,\mu;E^\star).
$$
We have to show that $F=0$. We first prove that for any $t>0$ we have 
 \begin{equation}
 \label{e6.6}
 \lim_{n\to\infty}D(P_t\varphi_n)(x)=\E[(Du^x(t))^*F(u^x(t))]\quad\mbox{\rm in }\;
L^2(E,\mu;E^\star).
 \end{equation} 
In fact, recalling Theorem \ref{bie34} and \eqref{bie11bis}, we have
$$
\begin{array}{l}
\ds\int_E|D(P_t\varphi_n)(x)-\E(Du^x(t))^*F(u^x(t))|^2_{E^\star}\,\mu(dx)\\
\\
=\int_E|\E\,Du^x(t)^*\le(D\varphi_n(u^x(t))-F(u^x(t))\r)|^2_{E^\star}\,\mu(dx)\\
\\
\ds\leq M\,e^{\omega t}\int_E\E\,|D\varphi_n(u^x(t))-F(u^x(t))|^2_{E^\star}\,\mu(dx)=M e^{\omega t}
\int_E|D\varphi_n(x)-F(x)|^2_{E^\star}\,\mu(dx),
\end{array}
$$
last inequality following from  the invariance of $\mu$. This implies \eqref{e6.6}.

Now, according to \eqref{e6.5} we have
$$
  \int_E(P_t\varphi_n)^2\mu(dx)+\int_0^tds\int_E \sum_{i=1}^\infty |\le<G(x)e_i,D(P_s\varphi_n)(x)\r>_E|^2\,\mu(dx)=\int_H\varphi_n^2(x)\mu(dx).
 $$
 Then,  we can take the limit as $n\to \infty$ in both sides and we get
 \[\lim_{n\to \infty}\int_0^tds\int_E \sum_{i=1}^\infty |\le<G(x)e_i,D(P_s\varphi_n)(x)\r>_E|^2\,\mu(dx)=0.\]
Due to \eqref{e6.6}, this implies that for any $i \in\,\nat$
\[\E\le<Du^x(t)G(x)e_i,F(u^x(t))\r>_E=0,\]
so that  
$$
\begin{array}{l}
P_t(\langle G(x)e_i,F(x)\rangle_E)=\E\langle G(u^x(t))e_i,F(u^x(t))\rangle_E\\
\\
=
\E\langle Du^x(t)G(x)e_i,F(u^x(t))\rangle_E+\E\langle  G(x)e_i-Du^x(t)G(x)e_i,F(u^x(t))\rangle_E\\
\\
+\E\langle  \le(G(u^x(t))-G(x)\r)e_i,F(u^x(t))\rangle_E\\
\\
=\E\langle  G(x)e_i-Du^x(t)G(x)e_i,F(u^x(t))\rangle_E
+\E\langle  \le(G(u^x(t))-G(x)\r)e_i,F(u^x(t))\rangle_E
\end{array}
$$ 
 Consequently, due to the continuity at $t=0$ of $u^x(t)$ and $Du^x(t)$, we get
$$
\lim_{t\to 0}P_t(\langle  G(\cdot)e_i,F\rangle_E)=0,\ \ \ \ \text{in}\ L^1(E,\mu).
$$
Since $P_t$ is a strongly continuous semigroup in $L^1(E,\mu)$, we deduce   $\langle  G(\cdot)e_i,F\rangle_E=0$ for all $i \in\,\nat$. As $G(x)$ is invertible and for any $h \in\,E$
\begin{equation}
\label{br10}
\lim_{n\to \infty}\sum_{i\leq n}\le<h,e_i\r>_He_i=h,\ \ \ \ \text{in}\ E,
\end{equation}
this  implies $\le<h,F(x)\r>_E=0$, for any $x,h \in\,E$, and then $F=0$.
\end{proof}

\medskip

Since $D$ is closable in $L^2(E,\mu)$, we define as usual the Sobolev space $W^{1,2}(E,\mu)$ as the domain of the closure of $D$ endowed with its graph norm. Notice that if  $\{\varphi_n\}\subset C^1_b(E)$ approximates some
 $\varphi \in\,W^{1,2}(E,\mu)$  in the graph norm of $D$, then, according to \eqref{br0}, the series
\[ \int_0^tds\int_E\sum_{i=1}^\infty\le|\le<G(x)e_i,D(P_s\varphi_n)(x)\r>_E\r|^2\,d\mu(x)\]
converges uniformly with respect to $n \in\,\nat$, so that \eqref{e6.5} holds for any $\varphi \in\,W^{1,2}(E,\mu)$.

\begin{Proposition}
Under Hypotheses \ref{H1}, \ref{H2} and \ref{H3}, for any $\varphi\in D(\mathcal K_2)$ we have
  \begin{equation}
  \label{e5.21bis}
  \int_E\mathcal K_2(\varphi)(x)\;\varphi(x)\;d\mu(x)=-\frac12\;\int_E\sum_{i=1}^\infty \le|\le<G(x)e_i,D\varphi(x)\r>_E\r|^2d\mu(x).
  \end{equation} 
 \end{Proposition} 
 \begin{proof}
 It follows from Lemma \ref{l5.8} and Proposition \ref{p6.5}.
 \end{proof}

 \subsection{The Poincar\'e inequality}

As a consequence of Hypothesis \ref{H3}, if we take $\a$ large enough, by proceeding as in Lemma \ref{Lbr24}  we have that there exists some $\theta>0$ such that
 \begin{equation}
 \label{e6.9}
 |DP_t\varphi(x)|_E\leq e^{-\theta t}\sup_{x \in\,E}|D\varphi(x)|_{E^\star}.
   \end{equation} 
By a standard argument this implies that for any $x \in\,E$
 \begin{equation}
 \label{e6.10}
 \lim_{t\to\infty}P_t\varphi(x)=\overline\varphi=\int_E\varphi d\mu.
 \end{equation}

 \begin{Proposition}
 \label{P6.7}
 Under Hypotheses \ref{H1}, \ref{H2} and \ref{H3},
there exist $\rho>0$ such that for all $\varphi\in W^{1,2}(E,\mu)$
 \begin{equation}
 \label{e6.12}
 \int_E|\varphi(x)-\overline\varphi|^2d\mu(x)\leq\rho\int_E|D\varphi(x)|_{E^\star}^2d\mu(x).
 \end{equation} 

 \end{Proposition}
\begin{proof}
We start from \eqref{e6.5} for $\varphi \in\,W^{1,2}(E,\mu)$,
  \[
  \int_E(P_t \varphi)^2(x) \mu(dx)+\int_0^tds\int_E \sum_{i=1}^\infty\le|\le<G(x)e_i,D(P_s\varphi)(x)\r>_E\r|^2\,\mu(dx)
  =\int_H\varphi^2(x)\mu(dx).
  \]
Taking into account of \eqref{bie1020tris}, this  yields
 $$
  \int_E(P_t \varphi)^2 \mu(dx)+c\int_0^tds\, e^{-\d s}s^{-\frac 12}\int_E P_s(|D\varphi(\cdot)|_{E^\star}^2)(x)\,\mu(dx)\geq \int_H\varphi^2(x)\mu(dx),
 $$
  which, by the invariance of $\mu$, yields
 $$
  \int_E(P_t \varphi)^2 \mu(dx)+c\int_0^tds\, e^{-\d s}s^{-\frac 12}\int_E |D\varphi(x)|_{E^\star}^2\,\mu(dx)\geq \int_H\varphi^2(x)\mu(dx),
 $$
  Letting $t\to \infty$, and recalling \eqref{e6.10}, this implies that for some $\rho>0$
  $$
  ( \overline \varphi)^2 +\rho\int_E |D\varphi(x)|_{E^\star}^2\,\mu(dx)\geq \int_H\varphi^2(x)\mu(dx),
  $$
  which is equivalent to \eqref{e6.12}.
\end{proof}

\subsection{Spectral gap and convergence to equilibrium}

\begin{Proposition}
\label{p6.7}
Under Hypotheses \ref{H1}, \ref{H2} and \ref{H3}, we have
$$
\sigma (\mathcal K_2)\backslash \{0\}\subset \left\{ \lambda \in \mathbb{C}:\;\frak{Re}\lambda \leq - \beta^2/\rho 
\right\},
$$
where $\sigma (\mathcal K_2)$ denotes the spectrum of  $\mathcal K_2$.

\end{Proposition}

 \begin{proof} Let us consider the space of all mean zero  functions from $L^{2}(E,\mu )$
\begin{displaymath}
L^{2}_{\pi}(E,\mu ):=\left\{\varphi \in L^{2}(E,\mu ):\; \overline{\varphi}=0 \right\}.
\end{displaymath}
Clearly
$$
L^{2}(E,\mu )=L^{2}_{\pi}(E,\mu )\oplus \reals.
$$
Moreover if $\overline{\varphi}=0$ we have by the invariance of $\mu$
$$\overline{(P_t\varphi)}=\int_HP_t\varphi(x)d\mu(x)=\int_H\varphi(x) d\mu(x)=0,
$$
so that $L^{2}_{\pi}(E,\mu )$ is an invariant subspace of $P_{t}.$

Denote by $\mathcal  K_\pi$ the restriction of $ \mathcal K_2$ to  $L^{2}_{\pi}(E,\mu )$. Then
we have clearly
$$
\sigma (\mathcal K_2)=\{0\}\cup \sigma (\mathcal  K_\pi).
$$
Moreover, if $\varphi \in L^{2}_{\pi}(E,\mu )$ we see, using  \eqref{e5.21}, that
\begin{equation}
\label{e3.4m}
\begin{array}{l}
\ds\int_{E}^{}\mathcal K_{\pi}\varphi(x) \;\varphi(x) \;d\mu(x)=\int_{E}^{}\mathcal K_2\varphi(x) \;\varphi(x) \;d\mu(x)
=-\frac{1}{2}\;\int_{E}\sum_{i=1}^\infty \le|\le<G(x)e_i,D\varphi(x)\r>_E \r|^{2}d\mu(x),
\end{array}
\end{equation}
Now, due to \eqref{br10}, for any $x, h \in\,E$ we have
\[\begin{array}{l}
\ds{\le|\le<G(x)h,D\varphi(x)\r>_E\r|^2=\le(\sum_{i=1}^\infty\le<G(x)e_i,D\varphi(x)\r>_E \le<h,e_i\r>_H\r)^2}\\
\vs
\ds{\leq \sum_{i=1}^\infty \le|\le<h,e_i\r>_H\r|^2\sum_{i=1}^\infty \le|\le<G(x)e_i,D\varphi(x)\r>_E\r|^2}
\end{array}\]
so that, as $|h|_E\leq 1$ implies $|h|_H\leq 1$,
\[|G^\star(x)D\varphi(x)|^2_{E^\star}\leq \sum_{i=1}^\infty \le|\le<G(x)e_i,D\varphi(x)\r>_E\r|^2.\]
Due to Hypothesis \ref{H2}, according to \eqref{e3.4m} this yields
\[\int_{E}^{}\mathcal K_{\pi}\varphi \;\varphi \;d\mu\leq -\frac {\beta^2}2\int_E|D\varphi(x)|_{E^\star}\,d\mu(x)\]
and by the Poincar\'e's inequality, we deduce
\begin{equation}
\label{e5m}
\begin{array}{l}
\ds\int_{E}^{}\mathcal K_{\pi }\varphi(x) \;\varphi(x) \;d\mu(x)
\leq -\frac{\beta^2}{2}\;\int_{E}^{}|D\varphi(x) |_{E^\star}^{2}d\mu(x) \leq -\frac{\beta^2}{2\rho}
\int_{E}^{}\varphi ^{2}(x)d\mu(x),
\end{array}
\end{equation}
which yields by the Hille--Yosida Theorem.
\begin{displaymath}
\sigma ( {\cal K}_{\pi })\subset \left\{ \lambda \in \C:\;\frak{Re}\,\lambda \leq -\beta^2 /2\rho  \right\}. \quad
\end{displaymath}
\end{proof}

 \begin{Remark}
{\em  The spectral gap implies the exponential convergence of $P_{t}\varphi $ to
 $\overline\varphi$. In fact  from
$$
\int_{E}^{}\mathcal K_{\pi }\varphi(x) \;\varphi(x) \;d\mu(x) \leq -\frac{\beta^2 }{2\rho}
\int_{E}^{}\varphi ^{2}(x)d\mu(x),
$$
we deduce that  $\mathcal K_{\pi }+\beta^2 /2\rho\,I$ is $m$-dissipative, so that
by  the Hille--Yosida Theorem we have
\begin{equation}
\label{e3.6m}
\int_E |P_t\psi(x)|^2d\mu(x)\leq e^{-\frac{\beta^2}\rho t}\int_E |\psi(x)|^2d\mu(x),\ \ \ \ \psi\in L^2_\pi(E,\mu).
\end{equation}
 Now given $\varphi\in L^2_\pi(E,\mu)$, setting in \eqref{e3.6m} $\psi:=\varphi-\overline\varphi$, we get
$$
\begin{array}{l}
\ds\int_E |P_t\varphi(x)-\overline\varphi|^2d\mu(x)\leq e^{-\frac{\beta^2}\rho t}\int_E |\varphi(x)-\overline\varphi|^2d\mu(x)\\
\\
\ds=e^{-\frac{\beta^2}\rho t}\left(\int_H \varphi^2(x)d\mu(x)-\overline\varphi^2\right)\leq e^{-\frac{\beta^2}\rho t}\int_E |\varphi(x)|^2d\mu(x).
\end{array}
$$}

$\Box$
\end{Remark}
 
\appendix
 
 \section{An It\^o formula in the space of continuos functions}

Fix $k\in\mathbb N$ and let  $b,\sigma_1,....\sigma_k$   be mappings from $H$ into $H$ and from $E$ into $E$,  which are  Lipschitz continuous    both in $H$ and in $E$. Let $X$ be the solution to  the stochastic differential equation
 \begin{equation}
 \label{eA.1}
 X(t)=x+\int_0^t b(X(s))ds+\sum_{i=1}^k\int_0^t\sigma_i(X(s))d\beta_i(s),
 \end{equation}
 where $\beta_1,...,\beta_n$ are independent real Brownian motions.\bigskip

   If $\varphi\in C^2_b(H)$, then it is well known that    the following It\^o's formula holds
      \begin{equation}
  \label{eA.2}
  \E \varphi(X(t))=\varphi(x)+\E\int_0^t \mathcal L\varphi(X(s))ds,
  \end{equation} 
  where  $\mathcal L$ is  the Kolmogorov operator  given by  
    \begin{equation}
  \label{eA.3}
\mathcal L\varphi(x)=\frac12\;\sum_{i=1}^k \langle D^2\varphi(x)\sigma_i(x),\sigma_i(x)\rangle_H+\langle D\varphi(x),b(x)\rangle_H,\ \ \ x \in\,H.
  \end{equation} \bigskip

Now we see what happens when dealing with \eqref{eA.2}  for functions defined in $E$.
  \begin{Proposition}
\label{pA.1}
  If $\varphi \in\,C^2_b(E)$, then it holds
 \begin{equation}
  \label{eA.4}  \E\varphi(X(t))=\varphi(x)+\E\int_0^t \mathcal L_E\varphi(X(s))ds,
  \end{equation} 
  where   $\mathcal L_E$ is given by
    \begin{equation}
  \label{eA.5}
 \mathcal L_E\varphi(x)=\frac12\;\sum_{i=1}^k \le<\sigma_i(x),D^2\varphi(x)\sigma_i(x)\r>_E+\le<b(x),D\varphi(x)\r>_E,
  \end{equation} 
  and $D_E$ represents the Fr\`echet derivative in $E$.\bigskip

\end{Proposition} 
\begin{proof}
In view of Lemma \ref{Lbr21}, if $\varphi \in\,C^2_b(E)$, there exists a sequence $\{\varphi_n\}_n \in\,\nat \subset C^1_b(H)$ such that
 $$
  \begin{array}{l}
\ds  \lim_{n\to\infty}\varphi_n(x)=\varphi(x),\ \ \ x\in E\\
\\
\ds  \lim_{n\to\infty}\langle y,D\varphi_n(x)\rangle_H =\le<y,D\varphi(x)\r>_E,\ \ \ x,y\in E\\
\\
\ds  \lim_{n\to\infty}\langle y,D^2\varphi_n(x)y\rangle_H =\le<y,D^2_E\varphi(x)y\r>_E,\ \ \ x,y\in E.
\end{array}
 $$
 Consequently
  \begin{equation}
  \label{eA.7}
  \lim_{n\to\infty}\mathcal  L\varphi_n(x)=\mathcal  L_E \varphi(x),\ \ \ x\in E.
  \end{equation} 

 Now,  by It\^o' s formula \eqref{eA.2}, we have for any $n \in\,\nat$ 
 \begin{equation}
  \label{eA.8}
  \E \varphi_n(X(t)) =\varphi_n(x)+\E\int_0^t \mathcal L\varphi_n(X(s))ds,
  \end{equation} 
and then, letting $n\to\infty$, we get \eqref{eA.4}.
\end{proof}

  \begin{Remark}
\label{rA.2}
\em Let $\varphi \in C^2_b(E)$. Then $\varphi^2 \in C^2_b(E)$ and we have
$$
\le<y,D_E\varphi^2(x)\r>_E=2\varphi(x)\,\le<y,D_E\varphi(x)\r>_E
$$
and
$$
\le<y,D^2_E\varphi^2(x)y\r>_E=2\varphi(x)\,\le<y,D^2_E\varphi(x)y\r>_E+2\le|\le<y,D_E\varphi(x)\r>_E\r|^2.
$$
Consequently
 \begin{equation}
  \label{eA.9}
\mathcal  L_E \varphi^2(x)=2 \varphi(x)\;\mathcal  L_E \varphi^2(x)+\sum_{k=1}^n\le|\le<\sigma_k(y),D_E\varphi(x)\r>_E\r|^2.
  \end{equation}

\end{Remark}


\begin{thebibliography}{99}

\bibitem{BE}  D. Bakry and M. \`Emery,  {\em  Hypercontractivity for diffusion semigroups}, C. R. Acad. Sci. Paris S\'er. I Math. 299 (1984), pp. 775--778. 

\bibitem{cerrai94}  S. Cerrai,  {\em A Hille--Yosida theorem for weakly
continuous semigroups},  Semigroup Forum 49 (1994),  pp. 349-367.

\bibitem{tesi} S. Cerrai, {\sc Second order PDE's in finite and infinite dimension. A probabilistic approach},
 Lecture Notes in Mathematics 1762, Springer
Verlag, 2001.


\bibitem{cerrai} S. Cerrai, {\em Stochastic reaction-diffusion systems with multiplicative noise and non-Lipschitz reaction term},  Probability Theory and Related Fields 125 (2003),  pp. 271--304.

\bibitem{stabi} S. Cerrai, {\em Stabilization by noise for a class of stochastic
reaction-diffusion equations}, Probability Theory and Related Fields  133 (2005),  pp. 190--214.

\bibitem{cerraisiam} S. Cerrai, {\em Averaging principle for systems of RDEs with polynomial nonlinearities perturbed by multiplicative noise}, Siam Journal of Mathematical Analysis 43 (2011),  pp. 2482-2518.

\bibitem{Abel}  G. Da Prato, {\em Kolmogorov equations for stochastic PDE's with multiplicative noise},
The Abel Symposium 2005, Stochastic Analysis and Applications (2007),  pp. 235--263.

\bibitem{DDG} G. Da Prato, A. Debussche and B. Goldys, {\em Invariant measures of non symmetric dissipative
 stochastic systems}, Probability Theory Related Fields  123 (2002),  pp. 355--380.


 \bibitem{DPZ3} G. Da Prato and J. Zabczyk, {\sc Stochastic equations in infinite 
dimensions,}
 Encyclopedia of Mathematics and its Applications,  Cambridge University
Press, 1992.



\bibitem{DPZ33} G. Da Prato and J. Zabczyk, {\sc Second Order Partial Differential Equations in Hilbert spaces,}  
London  Math. Soc.  Lecture Notes  293, Cambridge University Press, 2003. 


\bibitem{DS} J. D. Deuschel and D. Stroock,
  {\sc Large deviations,} Academic Press, 1984.
  
\bibitem{dmp} C. Donati-Martin, E. Pardoux, {\em White noise driven SPDEs with reflection},  Probability
Theory and Related Fields 95 (1993), pp. 1--24. 
   
    \bibitem{PZ} S. Peszat \& J. Zabczyk,  {\em Strong Feller property and
irreducibility for diffusions on Hilbert spaces,}
Annals of Probability 23 (1995), pp. 157--172.




\bibitem{Priola} E. Priola,   {\em On a class of Markov type semigroups
in spaces of uniformly continuous and bounded functions}, 
Studia Mathematica 136 (1999), pp. 271-295.



\end{thebibliography}
\end{document}